\documentclass[12pt]{amsart}
\usepackage{geometry}
\usepackage{hyperref}
                \usepackage{amsthm,amsfonts,amsmath,amscd,amssymb,epsfig,verbatim,
}

\usepackage{graphicx}
\usepackage{epstopdf}
\title{Stein structures: existence and flexibility}
\author{K.~Cieliebak and Y.~Eliashberg} 

\date{}  
\let\oldmarginpar\marginpar
\renewcommand\marginpar[1]{\-\oldmarginpar[\raggedleft\footnotesize #1]%
{\raggedright\footnotesize #1}}

\theoremstyle{plain}
\newtheorem{theorem}{Theorem}[section]
\newtheorem{thm}[theorem]{Theorem}

\newtheorem{cor}[theorem]{Corollary}

\newtheorem{prop}[theorem]{Proposition}
\newtheorem{lemma}[theorem]{Lemma}

\newtheorem{problem}[theorem]{Problem}
\theoremstyle{remark}

\newtheorem{remark}[theorem]{Remark}
\newtheorem*{remark*}{Remark}

\newtheorem*{example*}{Example}

\theoremstyle{definition}

\newcommand{\id}{{\rm id}}

\newcommand{\wt}{\widetilde}
\newcommand{\wh}{\widehat}

\newcommand{\p}{\partial}

\newcommand{\om}{\omega}

\newcommand{\eps}{\varepsilon}
\newcommand{\into}{\hookrightarrow}

\newcommand{\N}{{\mathbb{N}}}

\newcommand{\Z}{{\mathbb{Z}}}
\newcommand{\R}{{\mathbb{R}}}
\newcommand{\C}{{\mathbb{C}}}

\newcommand{\ind}{{\rm ind}}
\newcommand{\st}{{\rm st}}

\newcommand{\Int}{{\rm Int\,}} %Interior

\renewcommand{\min}{{\rm min}}
\renewcommand{\max}{{\rm max}}

\newcommand{\tb}{{\rm tb}}

\newcommand{\loc}{{\rm loc}}

\newcommand{\dist}{{\rm dist}}

\newcommand{\Imm}{\mathrm{Imm}}
\newcommand{\Mon}{\mathrm{Mon}}

\newcommand{\Hom}{\mathrm{Hom}}

\newcommand{\smoothmax}{\mathrm{smooth}\,\mathrm{max}}

\newcommand{\inj}{\mathrm{inj}}
\newcommand{\emb}{\mathrm{emb}}
\newcommand{\Emb}{\mathrm{Emb}}

\newcommand{\EE}{\mathcal{E}}

\newcommand{\PP}{\mathcal{P}}
\newcommand{\RR}{\mathcal{R}}

\newcommand{\II}{\mathcal{I}}

\def\Op{{\mathcal O}{\it p}\,}
\numberwithin{figure}{section}

\begin{document}

\begin{abstract}
This survey on the topology of Stein manifolds is an extract from the
book~\cite{CieEli12}.
% with Y.~Eliashberg. 
It is compiled from two short lecture series given by the first author
in 2012 at the Institute for Advanced Study, Princeton, and the
Alfr\'ed R\'enyi Institute of Mathematics, Budapest. 
\end{abstract}

\maketitle
%\tableofcontents

%%%%%%%%%%%%%%%%%%%%%%%%%%%%%%%%%%%%%%%%%%%%%%%%%%%%%%%%%%%%%%%%%%%%%
\section{The topology of Stein manifolds}\label{sec:conv}
%%%%%%%%%%%%%%%%%%%%%%%%%%%%%%%%%%%%%%%%%%%%%%%%%%%%%%%%%%%%%%%%%%%%%

%{\bf Stein manifolds. }
Throughout this article, $(V,J)$ denotes a smooth manifold
(without boundary) of real dimension $2n$ equipped with an {\em almost
  complex structure} $J$, i.e., an endomorphism $J:TV\to TV$ satisfying
$J^2=-\id$. The pair $(V,J)$ is called an {\em almost complex manifold}. 
It is called a {\em complex manifold} if the almost complex structure
$J$ is {\em integrable}, 
i.e., $J$ is induced by complex coordinates on $V$.
By the theorem of Newlander and Nirenberg~\cite{NewNir57}, 
a (sufficiently smooth) almost complex structure
$J$ is integrable if and only if its Nijenhuis tensor
$$
   N(X,Y):=[JX,JY]-[X,Y]-J[X,JY]-J[JX,Y], \qquad X,Y\in TV,
$$
vanishes identically. An integrable almost complex structure is called
a {\em complex structure}. 

A complex manifold $(V,J)$ is called {\em Stein} if it admits a proper
holomorphic embedding into some $\C^N$. Note that, due to the maximum
principle, every Stein manifold is open, i.e., it has no compact
components. 

By a theorem of Grauert, Bishop and
Narasimhan~\cite{Gra58,Bis61,Nar60}, a complex manifold 
$(V,J)$ is Stein if and only if it admits a smooth function
$\phi:V\to\R$ which is
\begin{itemize}
\item {\em exhausting}, i.e., proper and bounded from below, and
\item {\em $J$-convex} (or strictly plurisubharmonic),
  i.e., $-dd^\C\phi(v,Jv)>0$ for all $0\neq v\in TV$, where
  $d^\C\phi:=d\phi\circ J$. 
\end{itemize}
Note that the second condition means that $\om_\phi:=-dd^\C\phi$ is a
symplectic form compatible with $J$. Note also that the ``only if''
follows simply by restricting the $i$-convex function $\phi(z)=|z|^2$
on $\C^N$ (where $i$ denotes the standard complex structure) to a
properly embedded complex submanifold. Here are some examples of Stein
manifolds. 

(1) $(\C^n,i)$ is Stein, and properly embedded complex submanifolds of
Stein manifolds are Stein. 

(2) If $X$ is a closed complex submanifold of some projective space
$\C P^N$ and $H\subset\C P^N$ is a hyperplane, then $X\setminus H$ is 
Stein. 

(3) All open Riemann surfaces are Stein. 

(4) If $\phi:V\to\R$ is $J$-convex, then so is $f\circ\phi$ for any
smooth function $f:\R\to\R$ with $f'>0$ and $f''\geq 0$ (such 
$f$ will be called a {\em convex increasing function}). Given an
exhausting $J$-convex function $\phi:V\to\R$ and any $c\in\R$, we can
pick a diffeomorphism $f:(-\infty,c)\to\R$ with $f'>0$ and $f''\geq
0$; then $f\circ\phi$ is an exhausting $J$-convex function
$\{\phi<c\}\to\R$, hence the sublevel set $\{\phi<c\}$ is Stein.   

(5) Any strictly convex smooth function $\phi:\C^n\to\R$ is
$i$-convex. As a consequence, using (4), all convex open subsets of
$\C^n$ are Stein. 

(6) Let $L\subset V$ be a properly embedded {\em totally real}
submanifold, i.e., $L$ has real dimension $n$ and $T_xL\cap
J(T_xL)=\{0\}$ for all $x\in L$. Then 
the squared distance function $\dist_L^2:V\to\R$ from $L$ with respect
to any Hermitian metric on $V$ is $J$-convex on a neighbourhood of
$L$. As a consequence, $L$ has arbitrarily small Stein tubular
neighbourhoods in $V$ (which by (4) can be taken as sublevel sets
$\{\dist_L^2<\eps\}$ if $L$ is compact, but are more difficult to
construct if $L$ is noncompact). 

\begin{problem}\footnote{
``Problems'' in this survey are meant to be reasonably hard exercises
for the reader.}
Prove (1), (2), and the first statements in (4), (5), (6). 
\end{problem} 

\begin{problem}
A quadratic function
$\phi(z_1,\dots,z_n)=\sum_{j=1}^n(a_jx_j^2+b_jy_j^2)$ on $\C^n$ with
coordinates $z_j=x_j+iy_j$ is $i$-convex iff $a_j+b_j>0$ for all
$j=1,\dots,n$. A smooth function $\phi:\C\to\R$ is $i$-convex iff
$\Delta\phi>0$, i.e., $\phi$ is strictly subharmonic. 
\end{problem}

\begin{problem}
For an almost complex manifold $(V,J)$ define $\om_\phi:=-d(d\phi\circ
J)$ as in the integrable case. Then $\om_\phi(\cdot,J\cdot)$ is
symmetric for every function $\phi:V\to\R$ iff $J$ is integrable.  
\end{problem}
\medskip

%{\bf Topology of Stein manifolds. }
Let us now turn to the following question: {\it Which smooth manifolds
  $V$ admit the structure of a Stein manifold?} 

Clearly, one necessary condition is the existence of a (not
necessarily integrable) almost complex structure on $V$. This is a
topological condition on the tangent bundle of $V$ which can be
understood in terms of obstruction theory. For example, the odd
Stiefel-Whitney classes of $TV$ must vanish and the even ones must
have integral lifts.  

A second necessary condition arises from Morse theory.  
Recall that a smooth function $\phi:V\to\R$ is called {\em Morse} if
all its critical points are nondegenerate, and the {\em Morse index}
$\ind(p)$ of a critical point $p$ is the maximal dimension of a
subspace of $T_pV$ on which the Hessian of $\phi$ is negative
definite. The following simple observation, due to Milnor and
others, is fundamental for the topology of Stein manifolds. 

\begin{lemma}\label{lem:ind}
The Morse index of each nondegenerate critical point $p$ of a
$J$-convex function $\phi:V\to\R$ satisfies
$$
   \ind(p) \leq n = \dim_\C V.
$$
\end{lemma}

\begin{proof}\footnote{
``Proofs'' in this survey are only sketches of proofs; for details
see~\cite{CieEli12}.} 
Suppose $\ind(p)>n$. Then there exists a complex line $L\subset T_pV$ on 
which the Hessian of $\phi$ is negative definite. Pick a small
embedded complex curve $C\subset V$ through $p$ in direction $L$. Then
$\phi|_C$ has a local maximum at $p$, which contradicts the maximum
principle because $\Delta(\phi|_C)>0$.
%, i.e., $\phi|_C$ is strictly subharmonic.  
\end{proof}

This lemma imposes strong restrictions on the topology of Stein
manifolds: Consider a Stein manifold $(V,J)$ with exhausting
$J$-convex function $\phi:V\to\R$. After a $C^2$-small perturbation
(which preserves $J$-convexity) we may assume that $\phi$ is
Morse. Thus, by Lemma~\ref{lem:ind} and Morse theory, $V$ is obtained
from a union of balls by attaching handles $D^k\times D^{2n-k}_\eps$ of
indices $k\leq n$. In particular, all homology groups $H_i(V;\Z)$ with
$i>n$ vanish. 

Surprisingly, for $n>2$ these two necessary conditions are also
sufficient for the existence of a Stein structure:

\begin{thm}[\cite{Eli90}]\label{thm:ex}
A smooth manifold $V$ of real dimension $2n>4$ admits a Stein
structure if and only if it admits an almost complex structure $J$ and
an exhausting Morse function $\phi$ without critical points of index
$>n$. More precisely, $J$ is homotopic through almost complex
structures to a complex structure $J'$ such that $\phi$ is
$J'$-convex.  
\end{thm}

The idea of the proof is the following: Pick a sequence
$r_0<r_1<r_2<\dots$ of regular values of $\phi$ with $r_0<\min\,\phi$,
$r_i\to\infty$, and such that each interval $(r_i,r_{i+1})$ contains
at most one critical value of $\phi$. By Morse theory, each sublevel
set $W_i:=\{\phi\leq r_i\}$ is obtained from $W_{i-1}$ by attaching a
finite number of disjoint handles of index $\leq n$. Proceeding by
induction over $i$, suppose that on $W_{i-1}$, $J$ is already
integrable and $\phi$ is $J$-convex. Then for each $k\leq n$ we need
to 
\begin{enumerate}
\item extend $J$ to a complex structure over a $k$-handle, and
\item extend $\phi$ to a $J$-convex function over a $k$-handle. 
\end{enumerate}
The first step is based on $h$-principles and will be explained in
Section~\ref{sec:ex}. The second step requires the construction of
certain $J$-convex model functions on a standard handle and will be
explained in Section~\ref{sec:surr}.

%%%%%%%%%%%%%%%%%%%%%%%%%%%%%%%%%%%%%%%%%%%%%%%%%%%%%%%%%%%%%%%%%%%%%
\section{Constructions of $J$-convex functions}\label{sec:surr}
%%%%%%%%%%%%%%%%%%%%%%%%%%%%%%%%%%%%%%%%%%%%%%%%%%%%%%%%%%%%%%%%%%%%%

The goal of this section it to construct the $J$-convex model
functions needed for the proof of Theorem~\ref{thm:ex}. We begin with
some preparations. 

{\bf $J$-convex hypersurfaces. }
Consider a smooth hypersurface (of real codimension one) $\Sigma$ in a
complex manifold $(V,J)$. Each tangent space $T_p\Sigma\subset T_pV$,
$p\in\Sigma$, contains the unique maximal complex subspace
$\xi_p=T_p\Sigma\cap J(T_p\Sigma)\subset T_p\Sigma$. These subspaces
form a codimension one distribution $\xi\subset T\Sigma$, the 
{\em field of complex tangencies}. 
Suppose that $\Sigma$ is cooriented by a transverse vector field $\nu$
to $\Sigma$ in $V$ such that $J\nu$ is tangent to $\Sigma$.
The hyperplane field $\xi$ can be defined by a Pfaffian equation
$\{\alpha=0\}$, where the sign of the 1-form $\alpha$ is fixed by the
condition $\alpha(J\nu)>0$. The 2-form $\omega_\Sigma:=d\alpha|_\xi$,
called the {\em Levi form} of $\Sigma$, is then defined uniquely up to
multiplication by a positive function. The cooriented hypersurface
$\Sigma$ is called {\em $J$-convex} (or strictly Levi pseudoconvex) if
$\om_\Sigma(v,Jv)>0$ for each nonzero $v\in\xi$. 

\begin{problem}
Each regular level set of a $J$-convex function is $J$-convex (where
we always coorient level sets of a function by its gradient). 
Conversely, if $\phi:V\to\R$ is a smooth function without
critical points all of whose level sets are compact and $J$-convex,
then there exists a convex increasing function $f:\R\to\R$ such that
$f\circ\phi$ is $J$-convex. 
\end{problem}

Thus, up to composition with a convex increasing function, proper
$J$-convex functions are the same as {\em $J$-lc functions} (``lc''
stands for ``level convex''), i.e., functions that are $J$-convex near
the critical points and have compact $J$-convex level sets outside a
neighbourhood of the critical points. 

\begin{problem}\label{ex:complete}
Let $\phi:V\to\R$ be an exhausting $J$-convex function. Then for every
convex increasing function $f:\R\to\R$ with
$\lim_{y\to\infty}f'(y)=\infty$ the gradient vector field
$\nabla_{f\circ\phi}(f\circ\phi)$ is {\em complete}, i.e., its flow
exists for all times.  
\end{problem}
\medskip

{\bf Continuous $J$-convex functions. }
We will need the notion of $J$-convexity also for continuous
functions. To derive this, recall that $i$-convexity of a function
$\phi:U\to\R$ on an open subset $U\subset\C$ is equivalent to 
$\Delta\phi>0$. 

\begin{problem} 
A smooth function $\phi:U\to\R$ on an open subset $U\subset\C$
satisfies $\Delta\phi(z)\geq \eps>0$ at $z\in U$ if and only if 
it satisfies for each sufficiently small $r>0$ the {\em mean value
  inequality}  
\begin{equation} \label{eq:mean-value}
    \phi(z)+\frac{\eps r^2}{4} \leq
    \frac{1}{2\pi}\int_0^{2\pi}\phi(z+re^{i\theta})d\theta. 
\end{equation}
\end{problem}

Since inequality~\eqref{eq:mean-value} does not involve derivatives of
$\phi$, we can take it as the definition of $i$-convexity for a
continuous function $\phi:\C\supset U\to\R$, and hence via local
coordinates for a continuous function on a complex curve (note
however that the value $\eps$ depends on the local coordinate). 
%\marginpar{Explain!}
Finally, we call a continuous function $\phi:V\to\R$ on a complex
manifold {\em $J$-convex} if its restriction to every embedded
complex curve $C\subset V$ is $J$-convex. With this definition, we
have  

\begin{lemma}\label{lem:max}
The maximum $\max(\phi,\psi)$ of two continuous $J$-convex functions
is again $J$-convex.  
\end{lemma}

\begin{proof}
After restriction to complex curves it suffices to consider the case
$\phi,\psi:\C\supset U\to\R$. Then the mean value inequalities for
$\phi$ and $\psi$,
\begin{align*}
    \phi(z)+\frac{\eps_\phi r^2}{4} \leq
    \frac{1}{2\pi}\int_0^{2\pi}\phi(z+re^{i\theta})d\theta \leq
    \frac{1}{2\pi}\int_0^{2\pi}\max(\phi,\psi)(z+re^{i\theta})d\theta,
    \cr 
    \psi(z)+\frac{\eps_\psi r^2}{4} \leq
    \frac{1}{2\pi}\int_0^{2\pi}\psi(z+re^{i\theta})d\theta \leq
    \frac{1}{2\pi}\int_0^{2\pi}\max(\phi,\psi)(z+re^{i\theta})d\theta
\end{align*}
combine to the mean value inequality for $\max(\phi,\psi)$,
$$
    \max(\phi,\psi)(z)+\frac{\min(\eps_\phi,\eps_\psi) r^2}{4} \leq
    \frac{1}{2\pi}\int_0^{2\pi}\max(\phi,\psi)(z+re^{i\theta})d\theta. 
$$
\end{proof}

{\bf Smoothing of $J$-convex functions. }
Continuous $J$-convex functions are useful for our purposes because of

\begin{prop}[Richberg~\cite{Ric68}]\label{prop:Richberg}
Every continuous $J$-convex function on a complex manifold can be
$C^0$-approximated by smooth $J$-convex functions. 
\end{prop}

\begin{proof}
The proof is based on an explicit smoothing procedure for functions on
$\C^n$. Fix a smooth nonnegative function $\rho:\C^n\to\R$ with
support in the unit ball and $\int_{\C^n}\rho=1$. For $\delta>0$ set
$\rho_\delta(x):=\delta^{-2n}\rho(x/\delta)$. 
For a continuous function $\phi:\C^n\to\R$ define the 
``mollified'' function $\phi_\delta:\C^n\to\R$, 
\begin{equation}\label{eq:moll}
   \phi_\delta(x):=\int_{\C^n}\phi(x-y)\rho_\delta(y)d^{2n}y =
   \int_{\C^n}\phi(y)\rho_\delta(x-y)d^{2n}y.
\end{equation}
The last expression shows that the functions $\phi_\delta$ are smooth
for every $\delta>0$, and the first expression shows that
$\phi_\delta\to\phi$ as $\delta\to 0$ uniformly on compact
subsets. Moreover, if $\phi$ is $i$-convex, then the mean value
inequality for $\phi$ yields for all $x,w\in\C$ with $|w|$
sufficiently small
\begin{align*}
   \phi_\delta(x)+\frac{1}{4}\eps |w|^2 &=
   \int_{\C^n}\Bigl(\phi(x-y)+\frac{1}{4}\eps |w|^2\Bigr) \rho_\delta(y)
   d^{2n}y \cr
   &\leq \int_{\C^n}\frac{1}{2\pi}\int_0^{2\pi}\phi(x-y+we^{i\theta})d\theta
   \rho_\delta(y)d^{2n}y \cr
   &= \frac{1}{2\pi}\int_0^{2\pi}\phi_\delta(x+we^{i\theta})d\theta,
\end{align*}
so $\phi_\delta$ is $i$-convex. This proves the proposition on $\C^n$. 
The manifold case follows from this by a patching argument. 
\end{proof}

We will need four corollaries of Proposition~\ref{prop:Richberg}. The
first one is just combining it with Lemma~\ref{lem:max}: 

\begin{cor}[maximum construction for functions]\label{cor:max}
The maximum $\max(\phi,\psi)$ of two smooth $J$-convex functions
can be $C^0$-approximated by smooth $J$-convex functions. $\hfill\Box$  
\end{cor}

We will denote a smooth approximation of $\max(\phi,\psi)$ by
$\smoothmax(\phi,\psi)$. This is a slight abuse of notation because
such an approximation is not unique; it is somewhat justified by the
fact that the approximation can be chosen smoothly in families. 

\begin{cor}[interpolation near a totally real submanifold]
\label{cor:totally-real}
Let $L$ be a compact totally real submanifold of a complex manifold
$(V,J)$. Let $\phi,\psi:V\to\R$ be two smooth $J$-convex functions
such that $\phi(x)=\psi(x)$ and $d\phi(x)=d\psi(x)$ for all $x\in
L$. Then, given any neighborhood $U$ of $L$, there exists a smooth
$J$-convex function $\vartheta:V\to\R$ which coincides with $\phi$
outside $U$ and with $\psi$ in a smaller neighborhood of $L$.  
\end{cor}

\begin{proof}
For the construction, see Figure~\ref{fig:totally-real}. 
\begin{figure}
\centering
\includegraphics{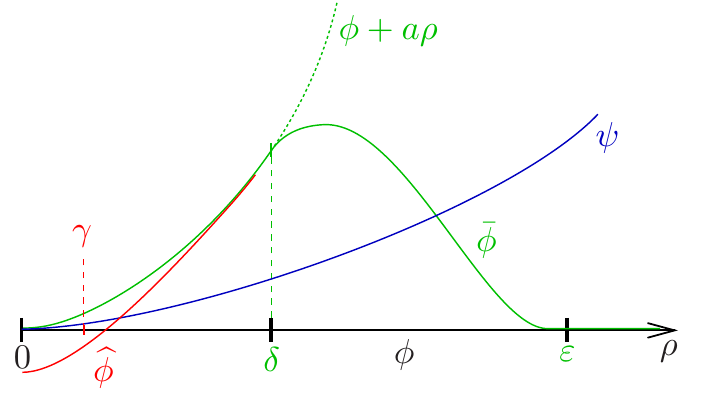}
\caption{Construction of the function $\vartheta$.}
\label{fig:totally-real}
\end{figure}
Shrink $U$ so that $\rho:=\dist_L^2:U\to\R$ is smooth and $J$-convex
and $U=\{\rho<\eps\}$.  
Since $\phi$ and $\psi$ agree to first order along $L$, we find an
$a>0$ such that $\phi+a\rho>\psi$ on $U\setminus L$. An explicit
computation shows that we can find a $J$-convex function $\bar\phi =
\phi+f(\rho)$ which agrees with $\phi$ outside $U$ and with
$\phi+a\rho$ on $\{\rho<\delta\}$ for some $\delta<\eps$. Perturb
$\bar\phi$ inside $\{\rho<\delta\}$ to a $J$-convex function $\wh\phi$
with $\wh\phi<\psi$ near $L$. Then the desired function $\vartheta$ is
given by $\smoothmax(\psi,\wh\phi)$ on $\{\rho<\delta\}$, and $\wh\phi$
outside.  
\end{proof}

\begin{cor}[minimum construction for hypersurfaces]\label{cor:min}
Let $\Sigma,\Sigma'$ be two compact $J$-convex hypersurfaces in a
complex manifold $(V=M\times\R,J)$ that are given as graphs of smooth 
functions $f,g:M\to\R$ and cooriented from below. Then there exists a
$C^0$-close smooth approximation of $\min(f,g)$
whose graph $\Sigma''$ is $J$-convex. 
\end{cor}

\begin{proof}
The functions $\phi(x,y):=y-f(x)$ and $\psi(x,y):=y-g(x)$ have
$J$-convex zero sets $\Sigma=\phi^{-1}(0)$ and
$\Sigma'=\psi^{-1}(0)$. Note that the zero set of
$\max(\phi,\psi)=y-\min(f,g)(x)$ is the graph of the function
$\min(f,g)$. Now pick a convex increasing function $h:\R\to\R$ with
$h(0)=0$ such that $h\circ\phi$ and $h\circ\psi$ are $J$-convex near
$\Sigma$ resp.~$\Sigma'$, and define $\Sigma''$ as the zero set of
$\smoothmax(h\circ\phi,h\circ\psi)$. 
\end{proof}

\begin{cor}[from families of hypersurfaces to foliations]
\label{cor:foliation}
Let $(M\times[0,1],J)$ be a compact complex manifold. Suppose there
exists a smooth family of $J$-convex graphs (cooriented from below)
$\Sigma_\lambda=\{y=f_\lambda(x)\}$, $\lambda\in[0,1]$, with
$\Sigma_0=M\times\{0\}$ and $\Sigma_1=M\times\{1\}$. Then there exists
a smooth {\em foliation} of $M\times[0,1]$ by $J$-convex graphs
$\wt\Sigma_\lambda=\{y=\wt f_\lambda(x)\}$ $\lambda\in[0,1]$, with
$\wt\Sigma_0=M\times\{0\}$ and $\wt\Sigma_1=M\times\{1\}$.
\end{cor}

\begin{proof}
By a family version of Corollary~\ref{cor:min}, 
the continuous functions $\bar f_\lambda:=\min_{\mu\geq\lambda}f_\mu$ can
be $C^0$-approximated by smooth functions $g_\lambda:M\to[0,1]$ whose
graphs $\{y=g_\lambda(x)\}$ are $J$-convex. Since 
$\bar f_\lambda\leq \bar f_{\lambda'}$ for $\lambda\leq\lambda'$, this
can be done in such a way that   
$g_\lambda\leq g_{\lambda'}$ for $\lambda\leq\lambda'$. So the graphs
of $g_\lambda$ almost form a foliation, and stretching them slightly
in the $y$-direction yields the desired foliation.  
\end{proof}

{\bf Open question. }
{\it Does an analogue of Proposition~\ref{prop:Richberg}, or at least of
Corollary~\ref{cor:max}, hold for non-integrable $J$? If this were
true, then a lot of the theory in these notes would work in the
non-integrable case.} 
\medskip

{\bf $J$-convex model functions. }
Let us fix integers $1\leq k\leq n$. 
Consider $\C^n$ with complex coordinates $z_j=x_j+iy_j$,
$j=1,\dots,n$, and set
$$
   R:=\sqrt{\sum_{j=1}^k x_j^2}, \qquad
   r:=\sqrt{\sum_{j=k+1}^n x_j^2+\sum_{j=1}^n y_j^2}.
$$ 
Fix some $a>1$ and define the {\em standard $i$-convex function}
$$
   \Psi_\st(r,R) := ar^2-R^2. 
$$
For small $\gamma>0$, we will use 
$$
   H_\gamma:=\{r\leq\gamma,\;R\leq 1+\gamma\}
$$
as a model for a complex $k$-handle. Its {\em core disk} is the
totally real $k$-disk $\{r=0,\;R\leq 1+\gamma\}$ and it will be
attached to the boundary of a Stein domain along the set
$\{r\leq\gamma,\;R=1+\gamma\}$. The following theorem will allow us to
extend a $J$-convex function over the handle. 

\begin{thm}\label{thm:model}
For each $0<\gamma<1<a$ there exists an $i$-lc function $\Psi(r,R)$ on
$H_\gamma$ with the following properties (see Figure~\ref{fig:model2}):
\begin{enumerate}
\item $\Psi=\Psi_\st$ near $\p H_\gamma$;
\item $\Psi$ has a unique index $k$ critical point at the origin;  
\item the level set $\Sigma=\{\Psi=-1\}$ {\em surrounds} the core disk
  in the sense that $\{r=0,\;R\leq 1+\gamma\}\subset\{\Psi<-1\}$. 
\end{enumerate}
\end{thm}

\begin{figure}
\centering
\includegraphics{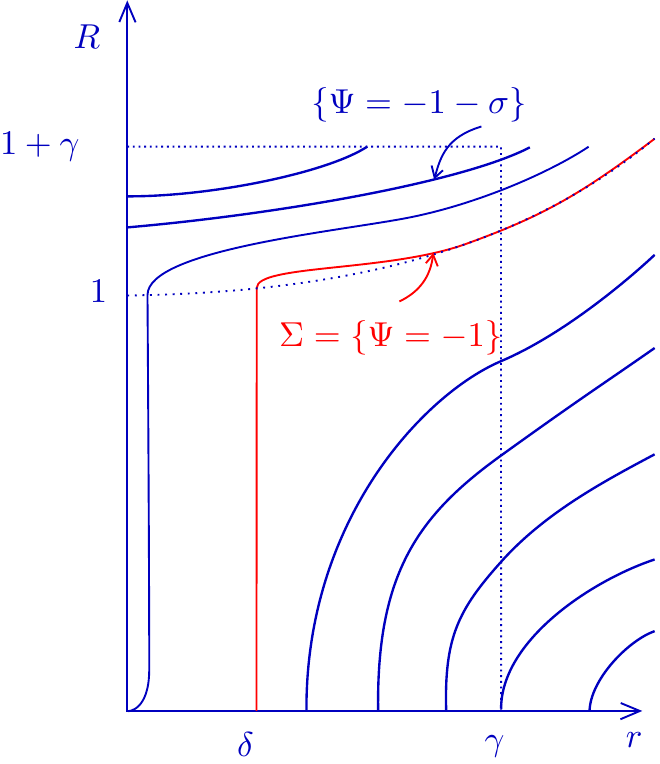}
\caption{The function $\Psi_1$.}
\label{fig:model2}
\end{figure}

\begin{proof}
{\bf Step 1. }
The first task is the construction of the hypersurface $\Sigma$. Let
us write $\Sigma$ as a graph $R=\phi(r)$, which we allow to become
vertical at $r=\delta$. One can work out the condition for
$i$-convexity of $\Sigma$ (cooriented from above), which becomes a
rather complicated system of second order differential inequalities
for $\phi$. However, it turns out that if $\phi>0$, $\phi'>0$, and
$\phi''\leq0$, the following simpler condition is {\em sufficient} for 
$i$-convexity: 
\begin{equation}\label{eq:shape-above}
   \phi''+\frac{\phi'^3}{r}-\frac{1}{\phi}(1+\phi'^2)>0.
\end{equation}

{\bf Step 2. }
To construct solutions of~\eqref{eq:shape-above}, we follow a
suggestion by M.~Struwe. We will find the function $\phi$ as a
solution of {\em Struwe's equation}
\begin{equation}\label{eq:Struwe}
   \phi''+\frac{\phi'^3}{2r}=0,
\end{equation}
with $\phi'>0$ and hence $\phi''<0$. Then~\eqref{eq:shape-above}
reduces to 
\begin{equation}\label{eq:Struwe-ineq}
   \frac{\phi'^3}{2r} - \frac{1}{\phi}(1+\phi'^2) > 0.
\end{equation}
Now Struwe's equation can be solved explicitly: 
It is equivalent to
$$
   \Bigl(\frac{1}{\phi'^2}\Bigr)'=-\frac{2\phi''}{\phi'^3}=\frac{1}{r},
$$
thus $1/\phi'^2=\ln(r/\delta)$ for some constant $\delta>0$, or
equivalently, $\phi'(r)=1/\sqrt{\ln(r/\delta)}$. By integration, this
yields a solution $\phi(r)$ for $r\geq\delta$ which is strictly
increasing and concave and satisfies $\phi'(\delta)=+\infty$. Choosing
the remaining integration constant appropriately, we find a solution
$\phi:[\delta,K\delta]\to\R$ which satisfies~\eqref{eq:Struwe-ineq}
and looks as shown in Figure~\ref{fig:Struwe}. Here $d>0$ can be
arbitrarily chosen and $K\delta$ can be made arbitrarily small. 
\begin{figure}
\centering
\includegraphics{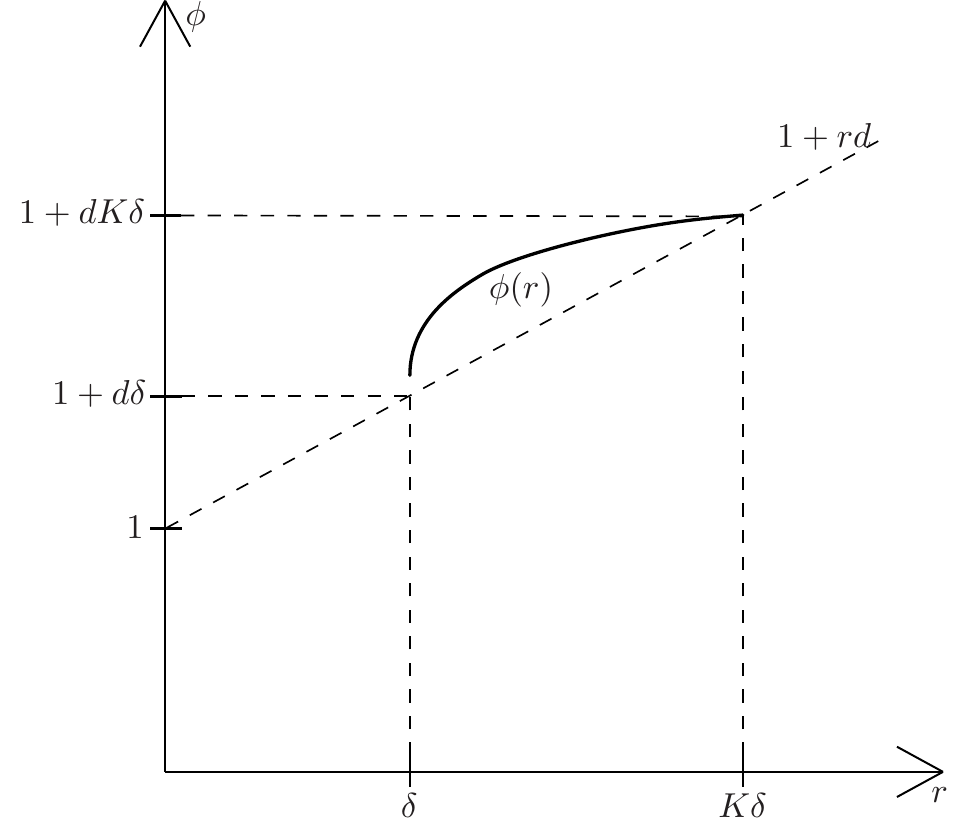}
\caption{A solution of Struwe's differential equation.}
\label{fig:Struwe}
\end{figure}

{\bf Step 3. }
Smoothing the maximum of the function $\phi$ from Step 2 and the
linear function $L(r)=1+dr$ yields an $i$-convex hypersurface which
surrounds the core disk and agrees with $\{R=L(r)\}$ for $r\geq
K\delta$. To finish the construction of the hypersurface $\Sigma$ in
Theorem~\ref{thm:model}, we still need to interpolate between $L(r)$
and the function $S(r)=\sqrt{1+ar^2}$ whose graph is the level set
$\{\Psi_\st(r,R) = ar^2-R^2 = -1\}$. Unfortunately, this cannot done
directly with the maximum construction because the graph of $L$ ceases
to define an $i$-convex hypersurface before it intersects the graph of
$S$. The solution is to interpolate from $L$ to a quadratic function
$Q(r)=1+br+cr^2/2$ and from there to $S$. The details are rather
involved due to the fact that the simple sufficient
condition~\eqref{eq:shape-above} fails and one needs to invoke the
full necessary and sufficient condition to ensure $i$-convexity during
this interpolation. 

{\bf Step 4. }
In Step 3 we constructed the level set $\Sigma$ as a graph
$\{R=\phi(r)\}$. To construct the $i$-lc function
$\Psi:H_\gamma\to\R$, in view of Corollary~\ref{cor:foliation} it
suffices to connect $\Sigma$ on both sides to level sets of $\Psi_\st$
by a smooth family of $i$-convex graphs. Towards larger $R$ this is a
simple application of the maximum construction, whereas towards
smaller $R$ it requires 1-parametric versions of the constructions in
Steps 1-3. This proves Theorem~\ref{thm:model}.
\end{proof}

%%%%%%%%%%%%%%%%%%%%%%%%%%%%%%%%%%%%%%%%%%%%%%%%%%%%%%%%%%%%%%%%%%%%%
\section{Existence of Stein structures}\label{sec:ex}
%%%%%%%%%%%%%%%%%%%%%%%%%%%%%%%%%%%%%%%%%%%%%%%%%%%%%%%%%%%%%%%%%%%%%

In this section we prove the Existence Theorem~\ref{thm:ex}.

{\bf Step 1: Extension of complex structures over handles. }\\
Consider an almost complex cobordism $(W,J)$ of complex dimension
$n\geq 1$ such that $J$ is integrable near $\p_-W$, and $\p_-W$ is
$J$-convex when cooriented by an inward pointing vector field. For
$k\leq n$ consider an embedding $f:(D^k,\p D^k)\into(W,\p_-W)$, where
$D^k\subset\R^k\subset\C^n$ is the closed unit disk.  

\begin{prop}\label{prop:J-ext}
The almost complex structure $J$ is homotopic rel $\Op(\p_-W)$ to one
which is integrable near $f(D^k)$. 
\end{prop}

\begin{proof}
After trivializing the relevant bundles, the differential of $f$
defines a map 
$$
   df:(D^k,\p D^k)\to (V_{2n,k},V_{2n-1,k-1}),
$$
where $V_{m,\ell}$ is the Stiefel manifold of $\ell$-frames in
$\R^m$. Let $V_{m,\ell}^\C\subset V_{2m,\ell}$ be the Stiefel manifold
of complex $\ell$-frames in $\C^m$, or equivalently, of totally real
$\ell$-frames in $\R^{2m}$. 

\begin{problem}
For each $n\geq 1$ and $k\leq n$, the map
$$
   \pi_k(V_{n,k}^\C,V_{n-1,k-1}^\C)\to \pi_k(V_{2n,k},V_{2n-1,k-1}) 
$$
induced by the obvious inclusions is surjective. 
\end{problem}

Thus there exists a homotopy $F_t:(D^k,\p D^k)\to
(V_{2n,k},V_{2n-1,k-1})$ from $F_0=df$ to some $F_1:(D^k,\p D^k)\to
(V_{n,k}^\C,V_{n-1,k-1}^\C)$. Now a relative version of {\em Gromov's
$h$-principle for totally real embeddings~\cite{Gro86,EliMis02}} yields
an isotopy of embeddings $f_t:(D^k,\p D^k)\into(W,\p_-W)$ from $f_0=f$
to a totally real embedding $f_1$. 

By a further isotopy we can achieve that $f_1|_{\p D^k}$ is real
analytic. We complexify $f_1|_{\p D^k}$ to a holomorphic embedding 
from a neighbourhood of $\p D^k$ in $\C^n$ into a slight extension
$\wt W$ of $W$ past $\p_-W$, and then extend it to an embedding $\wt
f_1:D^k\times D^{2n-k}_\eps\into \wt W$ which agrees 
with $f_1$ on $D^k=D^k\times 0$ and whose differential is complex
linear along $D^k$. The push-forward $(\wt f_1)_*i$ of the standard
complex structure $i$ on $D^k\times D^{2n-k}_\eps\subset\C^n$ agrees
with $J$ on a neighbourhood of $f_1(\p D^k)$ (since $\wt f_1$ is
holomorphic there) and at points of $f_1(D^k)$. Thus we can extend
$(\wt f_1)_*i$ to an almost complex structure $\wt J$ on $W$ which
coincides with $J$ near $\p_-W$ and outside a neighbourhood of
$f_1(D^k)$ and is integrable near $f_1(D^k)$. An application of the
isotopy extension theorem now yields the desired almost complex
structure which coincides with $J$ near $\p_-W$ and is integrable near
the original disk $f(D^k)$.  
\end{proof}

By induction over the handles, Proposition~\ref{prop:J-ext} yields the
following special case of the Gromov--Landweber theorem:

\begin{cor}[Gromov~\cite{Gro73}, Landweber~\cite{Lan74}]
Let $(V,J)$ be an almost complex manifold of complex dimension $n\geq
1$ which admits an exhausting Morse function $\phi:V\to\R$ without
critical points of index $>n$. Then $J$ is homotopic to an integrable
complex structure. \qed 
\end{cor}

{\bf Step 2: Extension of $J$-convex functions over handles. }\\
Consider again $(W,J)$ and $f:(D^k,\p D^k)\into(W,\p_-W)$ as in Step
1. After applying Proposition~\ref{prop:J-ext} we may assume that $J$
is integrable near $\Delta:=f(D^k)$. After real analytic approximation and
complexification, we may assume that $f$ extends to a holomorphic
embedding $F:H_\gamma\into \wt W$, where $H_\gamma$ is the standard
handle $D^k_{1+\gamma}\times D^{2n-k}_\gamma\subset\C^n$ and $\wt W$
is a slight extension of $W$ past $\p_-W$. 

Let $\phi$ be a given $J$-convex function near $\p_-W=\{\phi=-1\}$. To
finish the proof of Theorem~\ref{thm:ex}, we need to extend $\phi$ to
a $J$-convex function $\wt\phi$ on a neighbourhood of $\Delta$ whose
level set $\{\wt\phi=-1\}$ coincides with $\p_-W$ outside a
neighbourhood of $\p\Delta$ and surrounds $f(D^k)$ in $W$ as shown in
Figure~\ref{fig:J-convex-surr}. 
\begin{figure}
\centering
\includegraphics{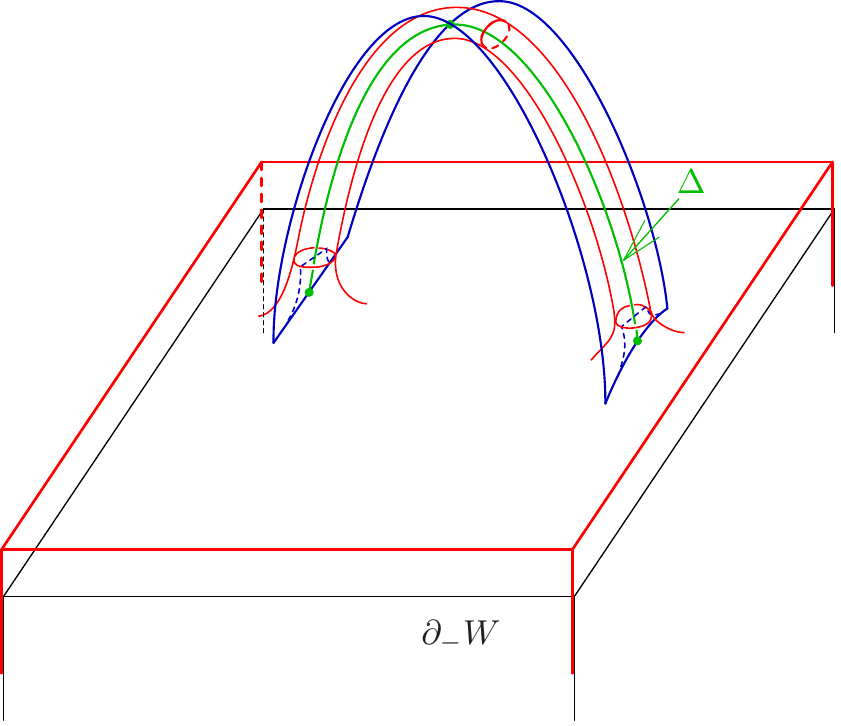}
\caption{Surrounding a $J$-orthogonally attached totally real disk.}
\label{fig:J-convex-surr}
\end{figure}

Equivalently, we need to extend $F^*\phi$ to
an $i$-convex function $\Psi$ on $H_\gamma$ whose
level set $\{\Psi=-1\}$ coincides with $\{F^*\phi=-1\}$ near $\p
H_\gamma$ and surrounds $D^k$ in $H_\gamma$. According to
Theorem~\ref{thm:model} in the previous section, this can be done if
we can arrange that $F^*\phi$ equals the standard function
$\Psi_\st(r,R)=ar^2-R^2$ near $\p D^k$. 

To analyze the last condition, note that the $n$-disk $D^n$ meets the
level set $\{\Psi_\st=-1\}$ {\em $i$-orthogonally} along $\p D^n$ in
the sense that $i(T_xD^n)\subset T_x\Sigma$ for all $x\in\p
D^n$. Conversely, suppose that $D^n$ is $i$-orthogonal to the level
set $\{F^*\phi=-1\}$ along $\p D^k$. Then $F^*\phi$ and $\Psi_\st$
have the same kernel $T_x\p D^n\oplus i(T_xD^n)$ at $x\in\p D^k$. 
After rescaling we may assume that $F^*\phi$ agrees with $\Psi_\st$ to
first order along $\p D^k$, so by Corollary~\ref{cor:totally-real} we
can deform $F^*\phi$ to make it coincide with $\Psi_\st$ near $\p
D^k$. 

The preceding discussion shows that it suffices to arrange that
$F(D^n\cap H_\gamma)$ is $J$-orthogonal to $\p_-W$ along
$\p\Delta=f(\p D^k)$.  
This can be arranged by appropriate choice of the extension $F$
provided that $\Delta$ is $J$-orthogonal to $\p_-W$ along
$\p\Delta$. Note that a necessary condition for this is $JT_x\p\Delta
\subset T_x\p_-W$ for $x\in \p\Delta$, which means that $\p\Delta$ is
{\em isotropic} for the contact structure $\xi=T\p_-W\cap 
J(T\p_-W)$ on $\p_-W$. Conversely, if this condition holds it is not
hard to arrange $J$-orthogonality. So we have reduced the proof of
Theorem~\ref{thm:ex} to 

\begin{prop}\label{prop:ex}
Consider an almost complex cobordism $(W,J)$ of complex dimension
$n$ such that $J$ is integrable near $\p_-W$, and $\p_-W$ is
$J$-convex when cooriented by an inward pointing vector field. If
$n > 2$, then any embedding $f:(D^k,\p D^k)\into(W,\p_-W)$, $k\leq
n$, is isotopic to one which is totally real on $D^k$ and isotropic on
$\p D^k$.  
\end{prop}

The remainder of this section is devoted to the proof of this
proposition. 

{\bf The subcritical case. }
Recall from Step 1 that there exists a homotopy $F_t:(D^k,\p D^k)\to
(V_{2n,k},V_{2n-1,k-1})$ from $F_0=df$ to some $F_1:(D^k,\p D^k)\to
(V_{n,k}^\C,V_{n-1,k-1}^\C)$. Restricting it to the boundary provides
a homotopy $G_t=F_t|_{\p D^k}:\p D^k\to V_{2n-1,k-1}$ from
$G_0=df|_{\p D^k}$ to some $G_1:\p D^k\to V_{n-1,k-1}^\C$.  
Now {\em Gromov's $h$-principle for isotropic
  immersions~\cite{Gro86,EliMis02}} yields a homotopy of immersions
$g_t:\p D^k\to \p_-W$ from $g_0=f|_{\p D^k}$ to an isotropic immersion
$g_1$ together with a 2-parameter family of maps $G_t^s:\p D^k\to
V_{2n-1,k-1}$ satisfying $G_t^0=dg_t$, $G_t^1=G_t$, $G_0^s=dg_0$, and
$G_1^s:\p D^k\to V_{n-1,k-1}^\C$ for all $s,t\in[0,1]$. 

If the $g_t$ can be chosen to be {\em embeddings} rather than
immersions, then the $h$-principle for totally real embeddings allows
us to extend the $g_t$ to embeddings $f_t:D^k\into W$ with $f_1$
totally real and the proposition follows. In the {\em subcritical}
case $k<n$, this can be achieved simply by a generic perturbation of
the $g_t$ (keeping $g_1$ isotropic). 

\begin{remark}
The existence of the 2-parameter family $G_t^s$ is crucial for the
application of the $h$-principle for totally real embeddings. Indeed,
we can always connect $g_0=f|_{\p D^k}$ by embeddings $g_t$ to some
isotropic embedding $g_1$, so if we could extend these $g_t$ to
totally real embeddings $D^k\into W$ we would prove
Proposition~\ref{prop:ex} also in the case $k=n=2$ where, as we shall
see below, it is false in general.  
\end{remark}

{\bf The critical case. }
In the {\em critical} case $k=n$, we can still perturb $g_1$ to a
Legendrian embedding, but the $g_t$ need not all be embeddings. To
understand the obstruction to this, consider the immersion
$$
   \Gamma:S^{n-1}\times[0,1]\to\p_-W\times[0,1],\qquad
   (x,t)\mapsto \bigl(g_t(x),t\bigr). 
$$ 
After a generic perturbation, we may assume that $\Gamma$ has finitely
many transverse self-intersections and define its {\em
  self-intersection index}
$$
   I_\Gamma := \sum_p I_\Gamma(p) \in 
   \begin{cases}
      \Z & \text{ if $n$ is even,} \\
      \Z_2 & \text{ if $n$ is odd}
   \end{cases}
$$
as the sum over the indices of all self-intersection points $p$. Here
the index $I_\Gamma(p)=\pm 1$ is defined by comparing the orientations
of the two intersecting branches of $\Gamma$ to the orientation of
$\p_-W\times[0,1]$. For $n$ even this does not depend on the order of
the branches and thus gives a well-defined integer, while for $n$ odd
it is only well-defined mod $2$. By a theorem of Whitney~\cite{Whi44},
for $n>2$, the regular homotopy $g_t$ can be deformed through regular
homotopies fixed at $t=0,1$ to an isotopy iff $I_\Gamma=0$. 

So if the family $g_t$ satisfies $I_\Gamma=0$ we are done. If
$I_\Gamma\neq 0$ we will connect $g_1$ to another Legendrian
embedding $g_2$ by a Legendrian regular homotopy $g_t$, $t\in[1,2]$,
whose self-intersection index equals $-I_\Gamma$. The extended family
$g_t$, $t\in[0,2]$, then has self-intersection index zero, so applying
the previous argument to this family will conclude the proof. 

{\bf Stabilization of Legendrian submanifolds. }
Consider a Legendrian submanifold $\Lambda_0$ in a contact manifold
$(M,\xi)$ of dimension $2n-1$. Near a point of $\Lambda_0$ pick
Darboux coordinates $(q_1,p_1,\dots,q_{n-1},p_{n-1},z)$ in which
$\xi=\ker(dz-\sum_jp_jdq_j)$ and the front projection of $\Lambda_0$
is a standard cusp $z^2=q_1^3$. Deform the two branches of the front
to make them parallel over some open ball
$B^{n-1}\subset\R^{n-1}$. After rescaling, we may thus assume that the
front of $\Lambda_0$ has two parallel branches $\{z=0\}$ and $\{z=1\}$
over $B^{n-1}$, see Figure~\ref{fig:stab}. 
\begin{figure}
\centering
\includegraphics{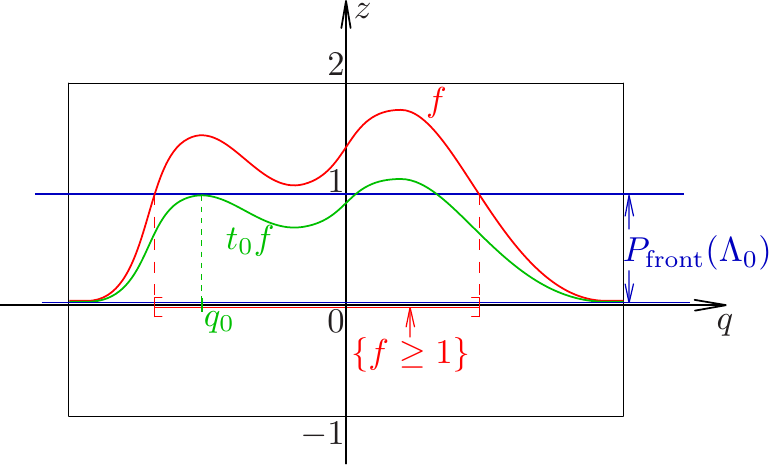}
\caption{Stabilization of a Legendrian submanifold.}
\label{fig:stab}
\end{figure}

Pick a non-negative function $f:B^{n-1}\to\R$ with compact support and
$1$ as a regular value, so $N:=\{f\geq 1\}\subset B^{n-1}$ is a
compact manifold with boundary. Replacing for each $t\in[0,1]$ the
lower branch $\{z=0\}$ by the graph $\{z=tf(q)\}$ of the function $tf$
yields the fronts of a path of Legendrian immersions $\Lambda_t\subset
M$ connecting $\Lambda_0$ to a new Legendrian submanifold $\Lambda_1$. 
Note that $\Lambda_t$ has a self-intersection for each critical point
of $tf$ on level $1$. 

\begin{problem}
The Legendrian regular homotopy $\Lambda_t$, $t\in[0,1]$, has
self-intersection index $(-1)^{(n-1)(n-2)/2}\chi(N)$. 
\end{problem}

\begin{problem}
For $n>2$ there exist compact submanifolds $N\subset\R^{n-1}$ of
arbitrary Euler characteristic $\chi(N)\in\Z$, while for $n=2$ the
Euler characteristic is always positive. 
\end{problem}

These two problems show that for $n>2$ the stabilization construction
allows us find a Legendrian regular homotopy $\Lambda_t$, $t\in[0,1]$,
with arbitrary self-intersection index. In view of the discussion above,
this concludes the proof of Proposition~\ref{prop:ex} and hence of
Theorem~\ref{thm:ex}. \hfill$\Box$
\medskip

\begin{remark}
The condition $n>2$ was used twice in the proof of
Proposition~\ref{prop:ex}: for the application of Whitney's theorem,
and to arbitrarily modify the self-intersection index by
stabilization. 
\end{remark}

To illustrate the failure of Theorem~\ref{thm:ex} for $n=2$, let us
analyze for which oriented plane bundles $V\to S^2$ the total space
admits a Stein structure. Here $V$ is oriented by minus the orientation of
the base followed by that of the fibre. Such bundles are classified by
their Euler class $e(V)$, which equals minus the self-intersection number
$S\cdot S\in\Z$ of the zero section $S\subset V$.  

We can construct each such bundle by attaching a 2-handle to the
4-ball $B^4$ along a topologically trivial Legendrian knot
$\Lambda\subset(S^3,\xi_\st)$. Let $\Delta\subset B^4$ be an embedded
2-disk meeting $\p B^4$ transversely along $\p\Delta=\Lambda$. It fits
together with the core disk $D$ of the handle to an embedded 2-sphere
$S\subset V$ giving the zero section in $V$. 
Recall that the {\em Thurston-Bennequin invariant} 
$\tb(\Lambda)$ is defined as the linking number of $\Lambda$ with a
push-off $\Lambda'$ in the direction of a Reeb vector field on
$(S^3,\xi_\st)$. 

\begin{problem}
The complex structure on $B^4\subset\C^2$ extends to a complex
structure on $V$ for which the core disk $D$ is totally real (and
hence by Theorem~\ref{thm:ex} to a Stein structure on $V$) iff $-e(V)
= S\cdot S = \tb(\Lambda)-1$.  
\end{problem}

In view of Bennequin's inequality $\tb(\Lambda)\leq -1$, this shows
that the construction of Theorem~\ref{thm:ex} works to provide a Stein
structure on $V$ iff $e(V)\geq 2$. A much deeper theorem of
Lisca and Mati\v{c}~\cite{LisMat97} 
%P.~Lisca and G.~Mati$\rm{\check{c}}$, 
(proved via Seiberg-Witten theory) asserts that $S\cdot
S\leq -2$ for every homologically nontrivial embedded 2-sphere $S$ in
a Stein surface, hence $V$ admits a Stein structure iff $e(V)\geq 2$. 
For example, the manifold $S^2\times\R^2$ does not admit any Stein
structure.

%%%%%%%%%%%%%%%%%%%%%%%%%%%%%%%%%%%%%%%%%%%%%%%%%%%%%%%%%%%%%%%%%%%%%
\section{Morse-Smale theory for $J$-convex
  functions}\label{sec:Morse-Smale} 
%%%%%%%%%%%%%%%%%%%%%%%%%%%%%%%%%%%%%%%%%%%%%%%%%%%%%%%%%%%%%%%%%%%%%

Morse-Smale theory deals with the problem of simplification of a Morse
function, trying to remove as many critical points as the topology
allows. One consequence is the $h$-cobordism theorem and the proof of
the higher-dimensional Poincar\'e conjecture. In this section we study
Morse-Smale theory for $J$-convex Morse functions, resulting in a
Stein version of the $h$-cobordism theorem. 
\medskip

{\bf The $h$-cobordism theorem. }
Let us begin by recalling the celebrated

\begin{theorem}[$h$-cobordism theorem,
  Smale~\cite{Sma62}]\label{thm:h-cobordism}
Let $W$ be an {\em $h$-cobordism}, i.e., a compact cobordism such that
$W$ and $\p_\pm W$ are simply connected and
$H_*(W,\p_-W;\Z)=0$. Suppose that $\dim W\geq 6$. 
Then $W$ carries a function without critical points and constant
on $\p_\pm W$. 
\end{theorem}

For the proof, one considers a compact cobordism $W$ with a Morse
function $\phi:W\to\R$ having $\p_\pm W$ as regular level sets and a
gradient-like vector field $X$ for $\phi$. We will refer to such
$(W,X,\phi)$ as a {\em Smale cobordism}. It is called {\em elementary}
if $W_p^-\cap W_q^+=\emptyset$ for all critical points $p\neq q$,
where $W_p^-$ and $W_p^+$ denotes the stable resp.~unstable manifold
of $p$ with respect to $X$. 

The key geometric ingredients in the proof of the $h$-cobordism 
theorem are the following four geometric lemmas about
modifications of Smale cobordisms (see~\cite{Mil65}). The first three
of them are rather simple, while the fourth one is more difficult.  

\begin{lemma}[moving critical levels]\label{lm:reordering}
Let $(W,X,\phi_0)$ be an elementary Smale cobordism. Then there exists a
homotopy $(W,X,\phi_t)$ of elementary Smale cobordisms
which arbitrarily changes the ordering of the values of the critical
points. 
\end{lemma}

\begin{lemma}[moving attaching spheres]\label{lm:moving}
Let $(W,X_0,\phi)$ be a Smale cobordism and $p\in W$ a critical point
whose stable manifold $W_p^-(X_0)$ with respect to $X_0$ intersects
$\p_-W$ along a sphere 
$S_0\subset \p_-W$. Then given any isotopy $S_t\subset \p_-W$,
$t\in[0,1]$, there exists a homotopy of Smale cobordisms $(W,X_t,\phi)$
such that the stable manifold $W^-_p(X_t)$ 
intersects $\p_-W$ along $S_t$. 
\end{lemma}

\begin{lemma}[creation of critical points]\label{lm:smooth-creation}
Let $(W,X_0,\phi_0)$ be a Smale cobordism without critical points. Then
for any $1\leq k\leq\dim W$ and any $p\in\Int W$ there exists a 
Smale homotopy $(W,X_t,\phi_t)$, $t\in[0,1]$, fixed outside a
neighbourhood of $p$, which creates a
pair of critical points of index $k-1$ and $k$ connected by a unique
trajectory of $X_1$ along which the stable and unstable manifolds
intersect transversely.  
\end{lemma}

\begin{lemma}[cancellation of critical points]
\label{lm:smooth-cancellation}
Suppose that a Smale cobordism $(W,X_0,\phi_0)$ contains exactly two
critical points of index $k-1$ and $k$ which are connected by a unique
trajectory of $X$ along
which the stable and unstable manifolds intersect transversely. 
Then there exists a Smale homotopy $(W,X_t,\phi_t)$, $t\in[0,1]$,
which kills the critical points, so the cobordism $(W,X_1,\phi_1)$ has
no critical points.  
\end{lemma}

Here all the homotopies will be fixed on a neighbourhood of $\p_\pm
W$. The functions $\phi_t$ in Lemmas~\ref{lm:smooth-creation}
and~\ref{lm:smooth-cancellation} will be Morse except for one value
$t_0\in(0,1)$ where they have a birth-death type critical point. 
Here a {\em birth-death type} critical point of index $k-1$ at $t_0$ 
is described by the local model
$$
   \phi_t(x) = x_1^3\mp (t-t_0)x_1 - x_2^2-\dots-x_k^2 +
   x_{k+1}^2+\dots + x_m^2.
$$

\begin{problem}
Prove Lemmas~\ref{lm:reordering}, \ref{lm:moving}
and~\ref{lm:smooth-creation}. 
\end{problem}
\medskip

{\bf Modifications of $J$-convex Morse functions. }
Let us now state the analogues of the four lemmas for $J$-convex
functions. By a {\em Stein cobordism} $(W,J,\phi)$ we will mean a
complex cobordism $(W,J)$ with a $J$-convex Morse function
$\phi:W\to\R$ having $\p_\pm W$ as regular level sets. We will always
use the gradient vector field $\nabla_\phi\phi$ of $\phi$ with respect
to the metric $g_\phi=-dd^\C\phi(\cdot,J\cdot)$ to obtain a Smale
cobordism $(W,\nabla_\phi\phi,\phi)$. Note that in the following four
propositions the complex structure $J$ is always fixed. 

\begin{prop}[moving critical levels]\label{prop:reordering}
Let $(W,J,\phi_0)$ be an elementary Stein cobordism. Then there exists a
homotopy $(W,J,\phi_t)$ of elementary Stein cobordisms
which arbitrarily changes the ordering of the values of the critical
points. 
\end{prop}

\begin{prop}[moving attaching spheres]\label{prop:moving}
Let $(W,J,\phi_0)$ be a Stein cobordism and $p\in W$ a critical point
whose stable manifold $W_p^-(\phi_0)$ with respect to
$\nabla_{\phi_0}\phi_0$ intersects $\p_-W$ along an isotropic
sphere $S_0\subset \p_-W$. Then given any {\em isotropic} isotopy
$S_t\subset \p_-W$, $t\in[0,1]$, there exists a homotopy of Stein
cobordisms $(W,J,\phi_t)$ with fixed critical point $p$ such that the
stable manifold $W^-_p(\phi_t)$ intersects $\p_-W$ along $S_t$.
\end{prop}

\begin{prop}[creation of critical points]\label{prop:creation}
Let $(W,J,\phi_0)$ be a Stein cobordism without critical points. Then
for any $1\leq k\leq\dim_\C W$ and any $p\in\Int W$ there exists a 
Stein homotopy $(W,J,\phi_t)$, $t\in[0,1]$, fixed outside a
neighbourhood of $p$, which creates a
pair of critical points of index $k-1$ and $k$ connected by a unique
trajectory of $\nabla_{\phi_1}\phi_1$ along which the stable and
unstable manifolds intersect transversely.  
\end{prop}

\begin{prop}[cancellation of critical points]
\label{prop:cancellation}
Suppose that a Stein cobordism $(W,J,\phi_0)$ contains exactly two
critical points of index $k-1$ and $k$ which are connected by a unique
trajectory of $\nabla_{\phi_0}\phi_0$ along
which the stable and unstable manifolds intersect transversely. 
Then there exists a Stein homotopy
$(W,J,\phi_t)$, $t\in[0,1]$, which kills the critical points, so the
cobordism $(W,J,\phi_1)$ has no critical points. 
\end{prop}

Again, all the homotopies will be fixed on a neighbourhood of $\p_\pm
W$, up to composition of the $J$-convex functions with some convex
increasing function $\R\to\R$. 
%Note that in these four propositions the complex structure $J$ is
%always fixed. 
The statements are precise analogues of those in the
smooth case, with one notable difference: in
Proposition~\ref{prop:moving} we require the isotopy $S_t$ to be {\em
  isotropic}. This difference, and the lack of a 1-parametric
$h$-principle for Legendrian embeddings, is largely responsible for all
symplectic rigidity phenomena on Stein manifolds. However, in the {\em
  subcritical case} $\ind(p)=k<n$ we have an $h$-principle stating
that any smooth isotopy $S_t$ starting at an isotropic embedding $S_0$
can be $C^0$-approximated by an isotropic isotopy starting at
$S_0$. With this, the proof of the $h$-cobordism theorem goes through
for $J$-convex functions and we obtain

\begin{theorem}[Stein $h$-cobordism
  theorem]\label{thm:Stein-h-cobordism} 
Let $(W,J,\phi)$ be a {\em subcritical} Stein $h$-cobordism. Suppose
that $\dim_\C W\geq 3$.  
Then $W$ carries a $J$-convex function without critical points and
constant on $\p_\pm W$. 
\end{theorem}
 
Further implications of these results will be discussed in
Section~\ref{sec:flex}. The remainder of this section is devoted to
the proofs of Propositions~\ref{prop:reordering}
  to~\ref{prop:cancellation}. 

\begin{proof}[Proof of Proposition~\ref{prop:reordering}]
This is an immediate consequence of the $J$-convex model functions
constructed in Section~\ref{sec:ex}: Since the cobordism is
elementary, the stable manifolds of the critical points are disjoint
embedded disks. For each critical point $p$,
Theorem~\ref{thm:model} allows us to deform $\phi_0$ near $W_p^-$
such that for the new $J$-lc function the level set containing $p$
is connected to a level set of $\phi_0$ slightly above $\p_-W$. Now we
perform this operation for each critical point and choose the level
sets near $\p_-W$ to achieve any given ordering. 
\end{proof}

\begin{proof}[Proof of Proposition~\ref{prop:moving}]
Let $k:=\ind(p)\leq n$. We identify level sets of $\phi_0$ near
$\p_-W$ via Gray's theorem. Then we construct an isotopy of embedded
$k$-disks $D_t\subset W$ such that $D_0=W_p^-$, $D_t$ agrees with
$W_p^-$ near $p$, $\p D_t=S_t$, and $D_t$ intersects all level sets of
$\phi$ below $c:=\phi(p)$ transversely in isotropic
$(k-1)$-spheres; see Figure~\ref{fig:attaching-spheres}. 
\begin{figure}
\centering
\includegraphics{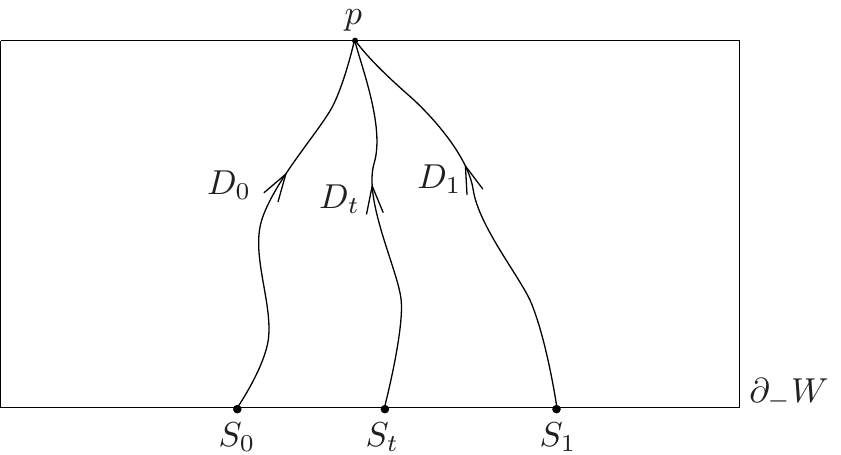}
\caption{Moving attaching spheres by isotropic isotopies.}
\label{fig:attaching-spheres}
\end{figure}
The last condition implies that $D_t$ is totally
real. If $k<n$ we can further extend $D_t$ to a totally real
embedding of $D^k\times D^{n-k}_\eps$ intersecting level sets
transversely in isotropic submanifolds, so it suffices to consider the
case $k=n$. To conclude the proof, we will construct $J$-convex
functions $\phi_t$ which agree with $\phi_0$ near $p$ and whose
gradient $\nabla_{\phi_t}\phi_t$ is tangent to $D_t$. This is done in
two steps. 

In the first step we construct $J$-convex functions $\psi_t$ whose
level sets below $c$ are $J$-orthogonal to $D_t$. To do this, consider
some level set $\Sigma$ of $\phi_0$ intersecting $D_t$ in the
isotropic submanifold $\Lambda_t$. Let $\xi$ be the induced contact
structure on $\Sigma$. We deform $\Sigma$ near $\Lambda_t$ to
a new hypersurface $\Sigma'$ which agrees with $\Sigma$ outside a
neighbourhood of $\Lambda_t$, intersects $D_t$ $J$-orthogonally in
$\Lambda_t$, and satisfies $\xi\subset T\Sigma'$ along $\Lambda_t$ (so
we ``turn $\Sigma$ around $\xi$ along $\Lambda_t$''); see
Figure~\ref{fig:turning-levels}.  
\begin{figure}
\centering
\includegraphics{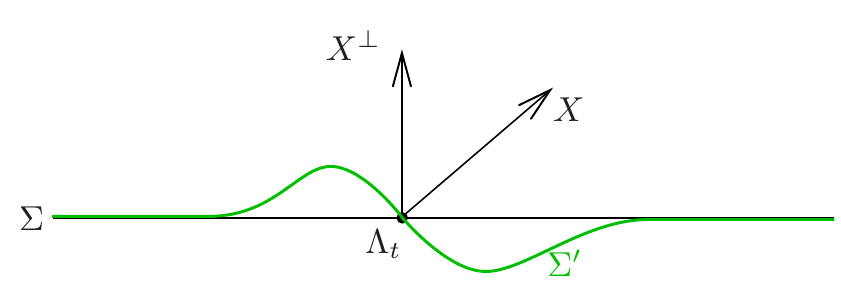}
\caption{Turning a $J$-convex hypersurface along an isotropic submanifold.}
\label{fig:turning-levels}
\end{figure}
A careful estimate of the Levi form shows that $\Sigma'$ can be made
$J$-convex. Deforming all level sets in this way leads to a family of
$J$-convex hypersurfaces, which by Corollary~\ref{cor:foliation} can
be turned into a foliation and thus into level sets of a $J$-lc
function. 

For the second step, consider the $J$-convex functions $\psi_t$ from
the first step whose level sets below $c$ are $J$-orthogonal to
$D_t$. It is not hard to write down in a local model a $J$-convex
function $\vartheta_t$ near $D_t$ which agrees with $\psi_t$ on $D_t$,
whose level sets are $J$-orthogonal to $D_t$, and whose
gradient $\nabla_{\vartheta_t}\vartheta_t$ is tangent to $D_t$. Now 
Corollary~\ref{cor:totally-real} provides the desired function
$\phi_t$ which coincides with $\psi_t$ outside a neighbourhood of
$D_t$ and with $\vartheta_t$ in a smaller neighborhood of $D_t$.   
\end{proof}

\begin{proof}[Proof of Proposition~\ref{prop:cancellation}]
Let $(W,J,\phi_0)$ be a Stein cobordism with exactly two
critical points $p$, $q$ of index $k$, $k-1$ connected by a unique
trajectory of $\nabla_{\phi_0}\phi_0$ along
which the stable and unstable manifolds intersect transversely. 
Set $a:=\phi_0|_{\p_-W}$, $b:=\phi_0(q)$ and $c:=\phi_0(p)$. 

\begin{problem}
In the situation above, suppose that $\phi_0$ is quadratic
in some holomorphic coordinates near $p$ and $q$. 
Then the closure of $W_p^-$ is an embedded $k$-dimensional half-disk
$\Delta\subset W$ with lower boundary $\p_-\Delta=\Delta\cap\p_-W$ and
upper boundary $\p_+\Delta=W_q^-$; see Figure~\ref{fig:cancellation}.
\end{problem}

\begin{figure}
\centering
\includegraphics{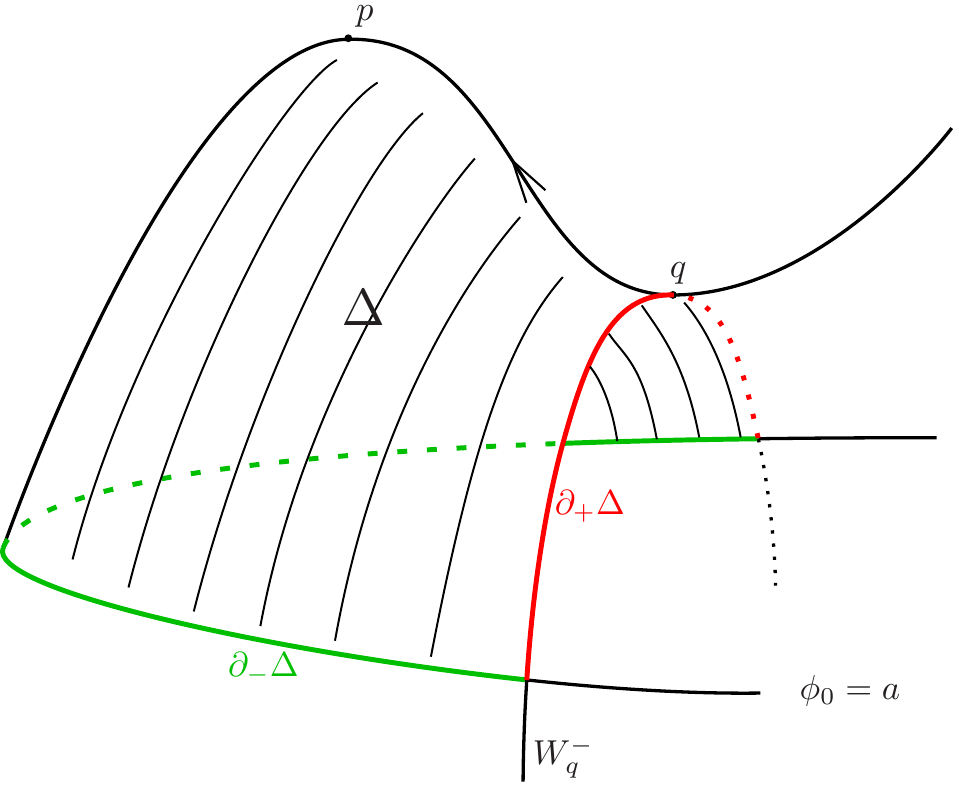}
\caption{The half-disk $\Delta$.}
\label{fig:cancellation}
\end{figure}

We will now deform the function $\phi_0$ in 4 steps. The first 3 steps
modify $\phi_0$ outside $\Delta$, without affecting its critical points,
to make some level set closely surround $\Delta$; the actual
cancellation happens in the last step.

{\bf First surrounding. }
First we apply Theorem~\ref{thm:model} to the $(k-1)$-disk
$\p_+\Delta$ to deform $\phi_0$ to a $J$-lc function $\phi_1$ such
that some level set $\Sigma_1=\{\phi_1=c_1\}$ closely surrounds
$\p_+\Delta$ as shown in Figure~\ref{fig:dumbell-2}. 
\begin{figure}
\centering
\includegraphics{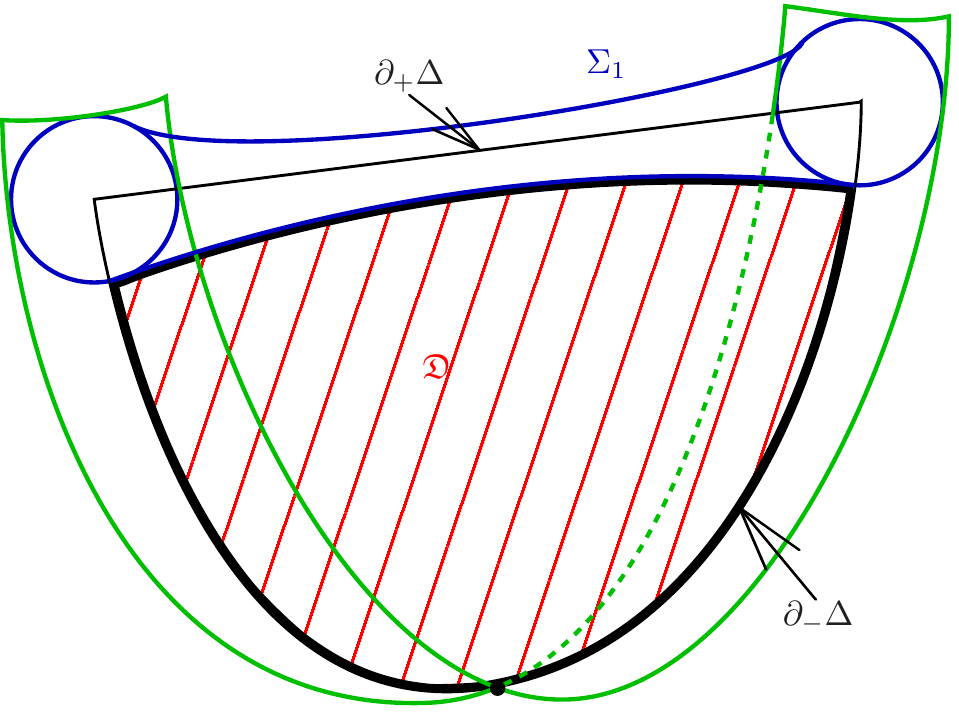}
\caption{The first surrounding hypersurface $\Sigma_1$ and the disk
  $\mathfrak{D}$.} 
\label{fig:dumbell-2}
\end{figure}

{\bf Second surrounding. }
Next we apply Theorem~\ref{thm:model} to the $k$-disk
$\mathfrak{D}:=\Delta\cap\{\phi_1\geq c_1\}$ to deform $\phi_1$ to a
$J$-lc function $\phi_2$ such that some level set
$\Sigma_2=\{\phi_2=c_2\}$ closely surrounds $\Delta$ as shown in
Figure~\ref{fig:dumbell-2}. Note that a cross-section of $\Sigma_2$
will have a dumbell-like shape as in Figure~\ref{fig:dumbell-1}, where
$x=(x_1,\dots,x_k)$ and $u=(x_{k+1},\dots,x_n,y_1,\dots,y_n)$.  
\begin{figure}
\centering
\includegraphics{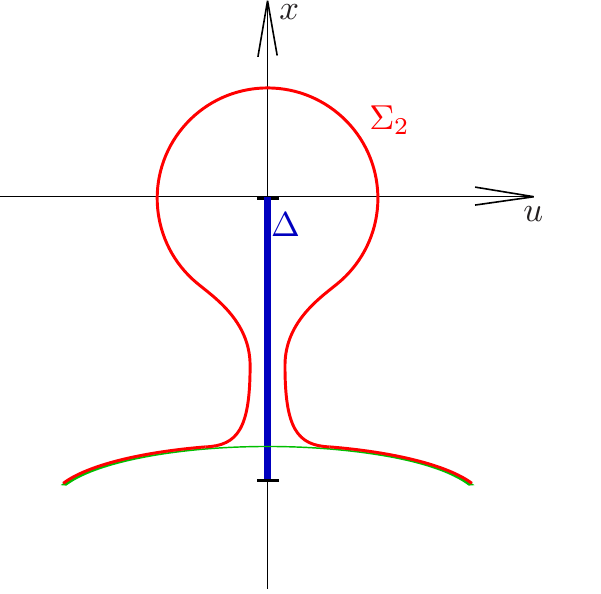}
\caption{The dumbell-shaped cross-section of the second surrounding
  hypersurface $\Sigma_2$.}
\label{fig:dumbell-1}
\end{figure}

{\bf Third surrounding. }
On the other hand, we can construct another hypersurface $\Sigma_3$
surrounding $\Delta$ as follows: 
Restrict a very thin model hypersurface $\Sigma$
provided by Theorem~\ref{thm:model} to a neighbourhood of the lower
half-disk $\{r=0,R\leq 1,y_k\leq 0\}$ in $\C^n$, implant it onto a
neighbourhood of $\Delta$ in $W$, and apply the minimum construction
in Corollary~\ref{cor:min} to this hypersurface and $\Sigma_2$.  
The resulting $J$-convex hypersurface $\Sigma_3$ is shown in
Figure~\ref{fig:dumbell-3}. 
\begin{figure}
\centering
\includegraphics{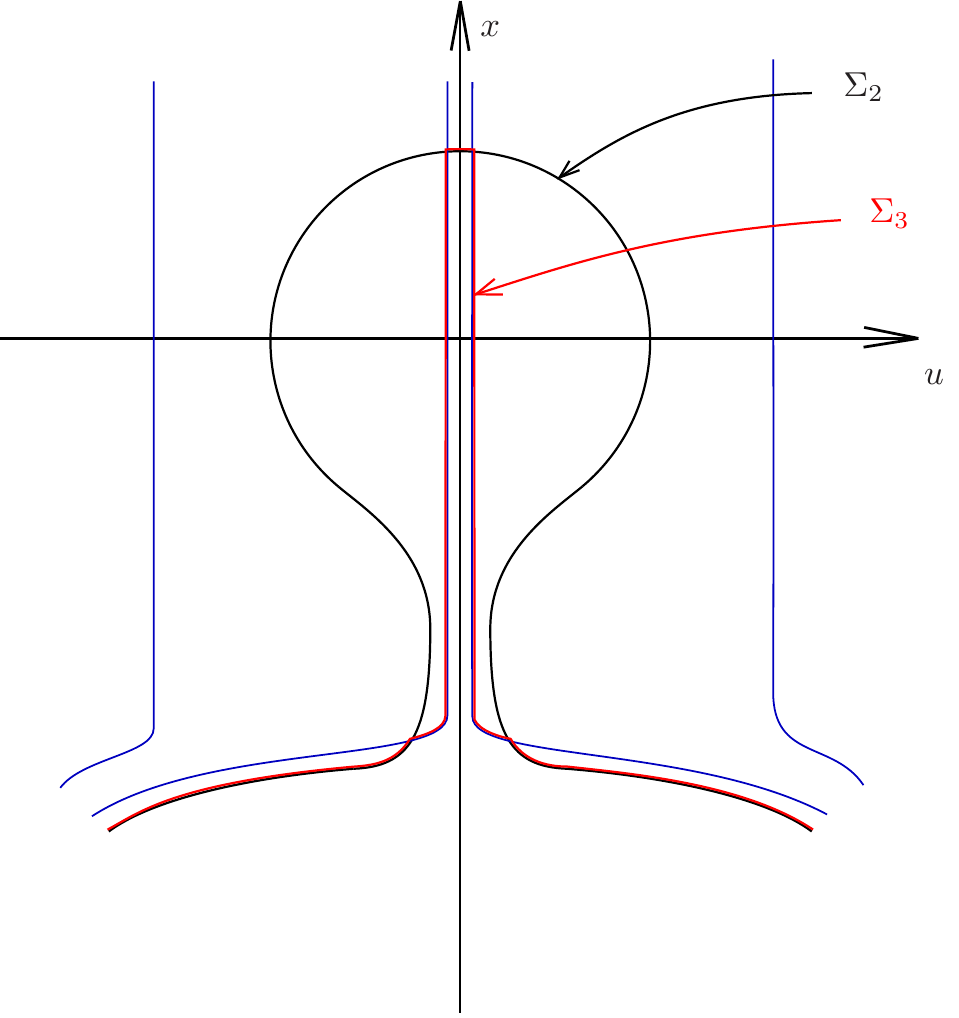}
\caption{The third surrounding hypersurface $\Sigma_3$.}
\label{fig:dumbell-3}
\end{figure}
The most difficult part is now to connect $\Sigma_3$ to $\Sigma_2$ by
a family of $J$-convex hypersurfaces. Once this is done, we can apply 
Corollary~\ref{cor:foliation} to deform $\phi_2$ to a $J$-lc function
$\phi_3$ having $\Sigma_3$ as a level set. 

{\bf The cancellation. }
Let us extend $\Delta$ across $\p_+\Delta$ to a slightly larger
half-disk $\Delta'$, still surrounded by $\Sigma_3$, so that the
critical points $p,q$ lie in the interior of $\Delta'$, and
$\nabla_{\phi_3}\phi_3$ is inward pointing along $\p_-\Delta'$ and
outward pointing along $\p_+\Delta'$. By
Lemma~\ref{lm:smooth-cancellation} there exists a family of smooth
functions $\beta_t:\Delta'\to\R$, $t\in[3,4]$, fixed near $\p\Delta'$,
such that $\beta_3=\phi_3|_{\Delta'}$ and $\beta_4$ has no critical
points. Identifying $\Delta'$ with the lower half-disk in the standard
handle, we can pick a large constant $B>0$ such that the functions
$\psi_t:=\beta_t+Br^2$ near $\Delta'$ are $J$-convex for all $t\in[3,4]$. 

After an application of Corollary~\ref{cor:totally-real}, we may
assume that $\psi_3=\phi_3$ near $\Delta'$. We can choose convex
increasing functions $f_t:\R\to\R$ with $f_3=\id$ such that for
$t\in[3,4]$ the $J$-convex function
$\phi_t:=\smoothmax(\psi_t,f_t\circ\phi_3)$ agrees with
$f_t\circ\phi_3$ in the region outside of $\Sigma_3$ and with $\psi_t$ near
$\Delta'$. In particular, $\phi_4$ has no critical points (for
this one needs to check that the maximum constructon does not create new
critical points outside $\Delta'$). Hence 
$(W,J,\phi_{4t})$, $t\in[0,1]$, is the desired Stein homotopy and
Proposition~\ref{prop:cancellation} is proved.  
\end{proof}

\begin{proof}[Proof of Proposition~\ref{prop:creation}]
The proof is similar to that of Proposition~\ref{prop:cancellation}
but much simpler. Set $a:=\phi_0|_{\p_-W}$ and $c:=\phi_0(p)$. 
Pick an isotropic embedded $(k-1)$-sphere $S$ through $p$ in the
level set $\phi_0^{-1}(c)$ and let $Z\subset W$ be the totally real
cylinder swept out by $S$ under the backward gradient flow of
$\phi_0$. We identify $Z$ with the cylinder $\{r=0,1/2\leq R\leq
1\}$ in the standard handle. A slight modification of
Theorem~\ref{thm:model} yields a family of $J$-convex functions
$\phi_t:W\to\R$, $t\in[0,1]$, such that some level set $\Sigma_1$ of
$\phi_1$ surrounds $Z$ in $W$. 

By Lemma~\ref{lm:smooth-creation} there exists a family of smooth
functions $\beta_t:Z\to\R$, $t\in[1,2]$, fixed near $\p Z$,
such that $\beta_1=\phi_1|_{Z}$ and $\beta_2$ has exactly two critical
points of index $k-1$ and $k$ connected by a unique gradient
trajectory along which the stable and unstable manifolds intersect
transversely. As above, we can pick a large constant $B>0$ such that
the functions $\psi_t:=\beta_t+Br^2$ near $Z$ are $J$-convex for all
$t\in[1,2]$ and set $\phi_t:=\smoothmax(\psi_t,f_t\circ\phi_1)$,
$t\in[1,2]$, to obtain the desired family $\phi_{2t}$, $t\in[0,1]$. 
\end{proof}

%%%%%%%%%%%%%%%%%%%%%%%%%%%%%%%%%%%%%%%%%%%%%%%%%%%%%%%%%%%%%%%%%%%%%
\section{Flexibility of Stein structures}\label{sec:flex}
%%%%%%%%%%%%%%%%%%%%%%%%%%%%%%%%%%%%%%%%%%%%%%%%%%%%%%%%%%%%%%%%%%%%%

In this section we study the question when two Stein structures
on the same manifold can be connected by a Stein homotopy.

{\bf Stein homotopies. }
Let us first carefully define the notion of a Stein homotopy. 
Consider first a smooth family (with respect to the
$C^\infty_\loc$-topology) of exhausting functions
$\phi_t:V\to\R$, $t\in[0,1]$, on a manifold $V$. We call it a {\em
  simple Morse homotopy} if there exists a family
of smooth functions $c_1<c_2<\dots$ on the interval $[0,1]$ such that
for each $t\in[0,1]$, $c_i(t)$ is a regular value of the function
$\phi_t$ and $\bigcup_k\{\phi_t\leq c_k(t)\}=V$. Then a  
{\em Morse homotopy} is a composition of finitely many simple Morse
homotopies, and a {\em Stein homotopy} is a family of Stein structures
$(V,J_t,\phi_t)$ such that the functions $\phi_t$ form a Morse
homotopy.  

%For simplicity, we will only consider Stein structures $(J,\phi)$ of {\em
%  finite type}, i.e., such that $\phi$ has only finitely many critical
%points. Then a (finite type) {\em Stein homotopy} between two finite
%type Stein structures $(J_0,\phi_0)$ and $(J_1,\phi_1)$ on the same
%manifold $V$ is a smooth family (with respect to the
%$C^\infty_\loc$-topology) of finite type Stein structures 
%$(J_t,\phi_t)$, $t\in[0,1]$, on $V$ such that all critical points of
%all the $\phi_t$ are contained in a compact subset of $V$. 
The role of the regular levels $c_i(t)$ is to prevent critical points
from ``escaping to infinity''. The following three problems motivate
why this is the correct definition. The first one shows that, without
this condition, the notion of ``homotopy'' would become rather
trivial: 

\begin{problem}
Any two Stein structures $(J_0,\phi_0)$ and $(J_1,\phi_1)$ on $\C^n$
can be connected by a smooth family of Stein structures
$(J_t,\phi_t)$ on $\C^n$, allowing critical points to escape to
infinity.   
\end{problem}

The second one shows that the question whether two Stein structures
are homotopic does not depend on the chosen $J$-convex functions:

\begin{problem}\label{ex:homotopy}
If $\phi_0,\phi_1:V\to\R$ are two exhausting $J$-convex functions for
the same complex structure $J$, then $(J,\phi_0)$ and $(J,\phi_1)$ can
be connected by a Stein homotopy $(J,\phi_t)$.   
\end{problem}

The third one makes the question of Stein homotopies accessible to
symplectic techniques.  
Let us call a Stein structure $(J,\phi)$ {\em complete} if the
gradient vector field $\nabla_\phi\phi$ is complete; by
Problem~\ref{ex:complete}, any Stein structure can be made complete
by composing $\phi$ with a convex increasing function $f:\R\to\R$. 

\begin{problem}\label{problem:symp}
If two complete Stein structures $(J_0,\phi_0)$ and
$(J_1,\phi_1)$ on a manifold $V$ are Stein homotopic, then the
associated symplectic manifolds $(V,-dd^\C\phi_0)$ and
$(V,-dd^\C\phi_1)$ are symplectomorphic.  
\end{problem}

From now on, when we talk about individual Stein structures $(J,\phi)$
we will always assume that the function $\phi$ is Morse, while for
Stein homotopies we allow birth-death type singularities. 
\medskip

{\bf The 2-index theorem. }
Before studying Stein homotopies, let us first consider the situation
in smooth topology. It follows from Problem~\ref{ex:homotopy} (simply
ignoring $J$-convexity) that
any two Morse functions on the same manifold can be connected by a
Morse homotopy. 
%Whether this can be done with critical points
%remaining in a compact set is already a more subtle question.
In addition, we will need some control over the indices of critical
points. This is provided by following immediate consequence of the 
{\em two-index theorem} of Hatcher and Wagoner
(\cite{HatWag73}, see also \cite{Igu88}):

\begin{thm}\label{thm:2-index}
%\marginpar{$m\geq 6$?}
Let $\phi_0,\phi_1:W\to[0,1]$ be two Morse functions on an
$m$-dimensional cobordism $W$ with $\p_\pm W$ as regular level sets.  
For some $k\geq 3$, suppose that $\phi_0,\phi_1$ have no critical points
of index $>k$. Then $\phi_0$ and $\phi_1$ can be connected by a
Morse homotopy $\phi_t$ (all having $\p_\pm W$ as regular level sets)
without critical points of index $>k$.  
\end{thm}
 
We will apply this theorem in the following two cases with $m=2n$:
\begin{itemize}
\item the subcritical case $k+1=n\geq 4$;
\item the critical case $k=n\geq 3$.
\end{itemize}
\medskip

{\bf Uniqueness of subcritical Stein structures. }
After these preparations, we can prove our first uniqueness theorem. 

\begin{thm}[uniqueness of subcritical Stein structures]
\label{subcrit-uniqueness}
Let $(J_0,\phi_0)$ and $(J_1,\phi_1)$ be two {\em subcritical} Stein
structures on the same manifold $V$ of complex dimension $n>3$. If
$J_0$ and $J_1$ are homotopic as almost complex structures, then
$(J_0,\phi_0)$ and $(J_1,\phi_1)$ are Stein homotopic. 
\end{thm}

\begin{proof}
By Theorem~\ref{thm:2-index} with $k+1=n\geq 4$, the functions
$\phi_0$ and $\phi_1$ can be connected a Morse homotopy $\phi_t$
without critical points of index $\geq n$. We cut the homotopy into a
finite number of simple Morse homotopies, and we cut each simple
homotopy at the regular levels $c_i$ into countably many compact
cobordisms. Let us pick gradient-like vector fields $X_t$ for
$\phi_t$. After further decomposition of these cobordisms,
we may assume that on each cobordism $W$ only one of the
following two cases occurs: 
\begin{enumerate}
\item all the Smale cobordisms $(W,X_t,\phi_t)$ are elementary; 
\item a pair of critical points is created or cancelled. 
\end{enumerate}
In the first case, only the levels of the critical points vary and the 
ataching spheres move by smooth isotopies. By the $h$-principle for
subcritical isotropic embeddings, these isotopies can be
$C^0$-approximated by isotropic isotopies. So we can apply
Propositions~\ref{prop:reordering} and~\ref{prop:moving} to realize
the same moves by $J$-convex functions. The second case is treated by
Propositions~\ref{prop:creation} and~\ref{prop:cancellation}. Applying
the four propositions inductively over the simple homotopies, and
within each simple homotopy over increasing levels, we hence construct
a family of $J_0$-convex functions (all for the same $J_0$!)
$\psi_t:V\to\R$ such that $\psi_t=\phi_t\circ h_t$ for a smooth family
of diffeomorphisms $h_t:V\to V$ with $h_0=\id$. 

Note that $\bigl((h_t)_*J_0,\phi_t\bigr)$ provides a Stein homotopy
from $(J_0,\phi_0)$ to $\bigl(J_2:=(h_1)_*J_0,\phi_1\bigr)$. So the
theorem is proved if we can connect $(J_2,\phi_1)$ to $(J_1,\phi_1)$
by a Stein homotopy $(J_t,\phi_1)$, $t\in[1,2]$ (with fixed function
$\phi_1$!). For this, we decompose $V$ into elementary cobordisms
containing only one critical level, and we pick a family $X_t$ of
gradient-like vector fields for $\phi_1$ connecting the gradients with
respect to $J_1$ and $J_2$. Then for each critical point $p$ on
such a cobordism $W$ the attaching spheres with respect to $X_t$ form
a smooth isotopy $S_t$, $t\in[1,2]$, connecting the isotropic spheres
$S_1$ and $S_2$. Again by the $h$-principle for subcritical isotropic
embeddings, we can make the isotopy $S_t$ isotropic. Now by a
1-parametric version of the Existence Theorem~\ref{thm:ex}, we can
connect $J_1$ and $J_2$ by a smooth family of integrable complex
structures $J_t$ on $W$ such that $\phi_1$ is $J_t$-convex for all
$t\in[1,2]$.  
\end{proof}

\begin{problem}
Find the major gap in the preceding proof, and consult~\cite{CieEli12}
on how it can be filled.  
\end{problem}

{\bf Exotic Stein structures. } 
In the critical case, uniqueness fails dramatically. In 2009,  
McLean~\cite{McL09} constructed infinitely many pairwise
non-homotopic Stein structures on $\C^n$ for any $n\geq 4$.
Extending McLean's result to $n=3$ (see~\cite{AboSei10}) and combining
it with the surgery exact sequence from~\cite{BEE09}, one obtains 

\begin{thm}\label{thm:exotic}
Let $(V,J)$ be an almost complex manifold of real dimension $2n\geq 6$
which admits an exhausting Morse function with finitely many critical
points all of which have index $\leq n$. Then $V$ carries infinitely
many finite type Stein structures $(J_k,\phi_k)$, $k\in\N$, such that
the $J_k$ are homotopic to $J$ as almost complex structures and
$(J_k,\phi_k)$, $(J_\ell,\phi_\ell)$ are not Stein homotopic for
$k\neq\ell$. 
\end{thm}

Here a Stein structure $(J,\phi)$ is said to be of {\em finite type}
if $\phi$ has only finitely many critical points. The Stein
structures $(J_k,\phi_k)$ are distinguished up to homotopy by showing
that the symplectic manifolds $(V,-dd^\C\phi_k)$ are pairwise
non-symplectomorphic, distinguished by their symplectic homology. 
Despite this wealth of exotic Stein structures, it has recently turned
out that there is still some flexibility in the critical case, which
we will describe next. 
\medskip

{\bf Murphy's h-principle for loose Legendrian knots. } 
It is well-known that the 1-parametric $h$-principle fails for
Legendrian embeddings. More precisely, a {\em formal Legendrian
isotopy} $(f_t,F^s_t)$ between two Legendrian embeddings
$f_0,f_1:\Lambda\into(M,\xi)$ consists of a smooth isotopy
$f_t:\Lambda\into M$, $t\in[0,1]$, together with a 2-parameter family
of injective bundle homomorphisms $F^s_t:T\Lambda\to TM$ covering
$f_t$, $s,t\in[0,1]$, such that $F^s_0=df_0$, $F^s_1=df_1$,
$F^0_t=df_t$, and $F^1_t:T\Lambda\to\xi$ is isotropic for all $s,t$. By the
$h$-principle for Legendrian immersions, this implies that $f_0$ and
$f_1$ are connected by a Legendrian regular homotopy. On the other
hand, there are many examples of pairs of Legendrian embeddings
that are formally Legendrian isotopic but not Legendrian isotopic (see
e.g.~\cite{Che02} in dimension 3, and~\cite{EES05} in higher dimensions). 

Despite the failure of the $h$-principle, there are two partial
flexibility results for Legendrian knots in dimension 3: Any two
formally isotopic Legendrian knots in $(\R^3,\xi_\st)$ become Legendrian
isotopic after sufficiently many stabilizations~\cite{FucTab97}, and
any two formally isotopic Legendrian knots in the complement of an
overtwisted disk are Legendrian isotopic~\cite{Dym01}. E.~Murphy
recently discovered a remarkable class of Legendrian embeddings in
dimensions $\geq 5$ which satisfy the 1-parametric $h$-principle:  

\begin{thm}[Murphy's $h$-principle for loose Legendrian embeddings~\cite{Mur11}]~~

\label{thm:loose}
In contact manifolds $(M,\xi)$ of dimension $\geq 5$ there exists a
class of {\em loose} Legendrian embeddings with the following
properties: 

(a) The stabilization construction described in Section~\ref{sec:ex}
with $\chi(N)=0$ turns any Legendrian embedding $f_0$ into a loose
Legendrian embedding $f_1$ formally isotopic to $f_0$.

(b) Let $(f_t,F_t^s)$, $s,t\in[0,1]$, be a formal Legendrian
isotopy connecting two loose Legendrian embeddings
$f_0,f_1:\Lambda\into M$. Then 
there exists a Legendrian isotopy $\wt f_t$ connecting  $\wt f_0=f_0$
and  $\wt f_1=f_1$ which is $C^0$-close to $f_t$ and is homotopic to
the formal isotopy $(f_t,F_t^s)$ through formal isotopies with
fixed endpoints. 
\end{thm}

Note that, in contrast to the 3-dimensional case, Legendrian
embeddings in dimension $\geq 5$ become loose after just {\em one}
stabilization, and the stabilization of a loose Legendrian embedding 
is Legendrian isotopic to the original one.  

{\bf Existence and uniqueness of flexible Stein structures. }
Let us call a Stein manifold $(V,J,\phi)$ of complex dimension $\geq
3$ {\em flexible} if if all attaching spheres on all regular level
sets are either subcritical or loose Legendrian. In view of
Theorem~\ref{thm:loose} (a), we can perform a stabilization in each
inductional step of the proof of the Existence Theorem~\ref{thm:ex} to
obtain 

\begin{thm}[existence of flexible Stein structures]
\label{thm:flexible-ex}
Any smooth manifold $V$ of dimension $2n>4$ which admits a Stein structure
also admits a {\em flexible} one (in a given homotopy class of almost
complex structures). \hfill$\Box$
\end{thm}

Now we can repeat the proof of Theorem~\ref{subcrit-uniqueness}, using
Theorem~\ref{thm:2-index} in the critical case $k=n\geq 3$ and
Theorem~\ref{thm:loose} (b) for the Legendrian attaching spheres
(always keeping the Stein structures flexible in the process), to
obtain  

\begin{thm}[uniqueness of flexible Stein structures]
\label{flexible-uniqueness}
Let $(J_0,\phi_0)$ and $(J_1,\phi_1)$ be two {\em flexible} Stein
structures on the same manifold $V$ of complex dimension $n>2$. If
$J_0$ and $J_1$ are homotopic as almost complex structures, then
$(J_0,\phi_0)$ and $(J_1,\phi_1)$ are Stein homotopic. \hfill$\Box$ 
\end{thm}

\begin{remark}
(a) Since subcritical Stein manifolds are flexible,
Theorem~\ref{flexible-uniqueness} allows us to weaken the hypothesis
on the dimension in Theorem~\ref{subcrit-uniqueness} from $n>3$ to
$n>2$.  

(b) Combining the result in~\cite{Cie02a} with the surgery exact sequence
in~\cite{BEE09} implies that flexible Stein manifolds have vanishing
symplectic homology. 
\end{remark}

{\bf Applications to symplectomorphisms and pseudo-isotopies. }
Theorem~\ref{flexible-uniqueness} has the following consequence for
symplectomorphisms of flexible Stein manifolds. 

\begin{theorem}\label{thm:symplectomorphism}
Let $(V,J,\phi)$ be a complete {\em flexible} Stein manifold of complex
dimension $n>2$, and $f:V\to V$ be a diffeomorphism such that $f^*J$ is
homotopic to $J$ as almost complex structures. 
Then there exists diffeotopy (i.e., a smooth family of diffeomorphisms)
$f_t:V\to V$, $t\in[0,1]$, such that $f_0=f$, and $f_1$ is a
symplectomorphism of $(V,\om_\phi)$.   
\end{theorem} 

\begin{proof}
By Theorem~\ref{flexible-uniqueness}, there exists a Stein homotopy
$(J_t,\phi_t)$ connecting the flexible Stein structures
$(J_0,\phi_0)=(J,\phi)$ and $(J_1,\phi_1)=(f^*J,f^*\phi)$. By
Problem~\ref{problem:symp}, there exists a 
diffeotopy $h_t:V\to V$ such that $h_0=\id$ and
$h_t^*\om_{\phi_t}=\om_\phi$. In particular,
$(f\circ h_1)^*\om_\phi=h_1^*\om_{\phi_1}=\om_\phi$, so $f_t=f\circ
h_t$ is the desired diffeotopy.   
\end{proof}

\begin{remark}
Even if $(J,\phi)$ is of finite type and $f=\id$ outside a compact
set, the diffeotopy $f_t$ provided by
Theorem~\ref{thm:symplectomorphism} will in general {\em not} equal
the identity outside a compact set. 
\end{remark}

For our last application, consider a closed manifold $M$. A {\em
  pseudo-isotopy} of $M$ is a smooth function $\phi:M\times[0,1]\to\R$ 
without critical points which is constant on $M\times 0$ and $M\times
1$ with $f|_{M\times 0}<f|_{M\times 1}$. 
We denote by $\EE(M)$ the space of pseudo-isotopies equipped with
the $C^\infty$-topology. The homotopy group $\pi_0\EE(M)$ is trivial
if $\dim M\geq 5$ and $M$ is simply connected~\cite{Cer70}, while in
the non-simply connected case for $\dim M\geq 6$ it is often
nontrivial~\cite{HatWag73,Igu88}.

\begin{problem}
Show that $\EE(M)$ is homotopy equivalent to the space $\PP(M)$ of
diffeomorphisms of $M\times[0,1]$ that restrict as the identity to
$M\times 0$. (The map $\PP(M)\to\EE(M)$ assigns to $f$ the pullback
$f^*\phi_\st$ of the function $\phi_\st(x,t)=t$, and a homotopy inverse
is obtained by following trajectories of a gradient-like vector
field). This explains the name ``pseudo-isotopy'' because any isotopy
$f_t:M\to M$ with $f_0=\id$ defines an element
$f(x,t)=\bigl(f_t(x),t\bigr)$ in $\PP(M)$.   
\end{problem}

Now consider a topologically trivial Stein cobordism
$(M\times[0,1],J,\phi)$ and denote by
$\EE(M\times[0,1],J)$ the space of $J$-convex functions
$M\times[0,1]\to\R$ without critical points which are constant on
$M\times 0$ and $M\times 1$ with $f|_{M\times 0}<f|_{M\times 1}$. 
%The proof of Theorem~\ref{flexible-uniqueness} shows that the space 
%$\EE(M\times[0,1],J)$ is non-empty whenever $\dim M>3$.  

\begin{thm}\label{thm:pseudo-Stein}
For any topologically trivial {\em flexible} Stein cobordism
$(M\times[0,1],J,\phi)$ of dimension $2n>4$ the canonical inclusion
$\II:\EE(M\times[0,1],J)\into\EE(M)$ induces a {\em surjection}
$$
   \II_*:\pi_0\EE(M\times[0,1],J)\to\pi_0\EE(M).
$$ 
\end{thm}

\begin{proof}
Let $\psi\in\EE(M)$ be given. By Theorem~\ref{thm:2-index} with
$k=n\geq 3$, there exists a Morse homotopy $\phi_t:M\times[0,1]\to\R$ 
without critical points of index $>n$ connecting the $J$-convex
function $\phi_0=\phi$ to $\phi_1=\psi$. Arguing as in the proof of
Theorem~\ref{subcrit-uniqueness}, always keeping the Stein
structures flexible, we construct a diffeotopy $h_t:M\times[0,1]\to
M\times[0,1]$ with $h_0=\id$ such that the functions
$\psi_t=\phi_t\circ h_t$ are $J$-convex for all $t\in[0,1]$. Then the
$J$-convex function $\psi_1=\psi\circ h_1$ is connected to $\psi$ by
the path $\psi\circ h_t$ of functions without critical points, so
$\psi_1$ and $\psi$ belong to the same path connected component of
$\EE(M)$.     
\end{proof}

We conjecture that $\II_*$ in Theorem~\ref{thm:pseudo-Stein} is an
isomorphism.

%%%%%%%%%%%%%%%%%%%%%%%%%%%%%%%%%%%%%%%%%%%%%%%%%%%%%%%%%%%%%%%%%%%%%
%%%%%%%%%%%%%%%%%%%%%%%%%%%%% REFERENCES %%%%%%%%%%%%%%%%%%%%%%%%%%%%
%%%%%%%%%%%%%%%%%%%%%%%%%%%%%%%%%%%%%%%%%%%%%%%%%%%%%%%%%%%%%%%%%%%%%

\end{document}